\documentclass[12pt,a4paper]{amsart}
\usepackage{graphicx}
\usepackage[usenames,dvipsnames]{xcolor}
\usepackage[colorlinks=true,linkcolor=Blue,citecolor=Green]{hyperref}
\usepackage{amsmath,amsthm,amsfonts,amssymb}
\usepackage[T1]{fontenc}
\usepackage{graphicx}
\usepackage{fullpage}
\usepackage{latexsym}
\usepackage{amsfonts}
\usepackage{amsmath}
\usepackage{amssymb}
\usepackage[all]{xy} 

\newtheorem{theorem}{Theorem}[section]
\newtheorem{lemma}[theorem]{Lemma}
\newtheorem{proposition}[theorem]{Proposition}

\newtheorem{corollary}[theorem]{Corollary}

\newtheorem{question}[theorem]{Question}

\theoremstyle{definition}
\newtheorem{definition}[theorem]{Definition}

\newtheorem{example}[theorem]{Example}
\newtheorem{remark}[theorem]{Remark}

\DeclareMathOperator{\Conv}{Conv}
\DeclareMathOperator{\Cone}{Cone}

\def\R{\mathbb{R}}
\def\S{\mathbb{S}}
\def\C{\mathbb{C}}

\def\Z{\mathbb{Z}}
\def\Q{\mathbb{Q}}

\def\N{\mathbb{N}}

\newcommand{\MM}{\mathcal{M}}
\newcommand{\CC}{\mathcal{C}}
\newcommand{\DD}{\mathcal{D}}
\newcommand{\PP}{\mathcal{P}}

\newcommand{\NN}{\mathcal{N}}
\newcommand{\VV}{\mathcal{V}}

\renewcommand{\leq}{\leqslant}
\renewcommand{\geq}{\geqslant}

\DeclareMathOperator{\distance}{dist}

\DeclareMathOperator{\Vol}{Vol}

\DeclareMathOperator{\lin}{lin}
\DeclareMathOperator{\GL}{GL}
\DeclareMathOperator{\AGL}{AGL}

\newcommand{\dist}{\delta}

\usepackage[colorinlistoftodos, textwidth=3.5cm,textsize=small]{todonotes}

\newcommand{\T}{{\mathbb T}}
 
\newcommand{\om}{\omega} 
\newcommand{\Mo}{\mathcal{M}(2n)}
\newcommand{\Mt}{\widetilde{\mathcal{M}(2n)}}
\newcommand{\Dtt}{\widetilde{\mathcal{D}(n)}}
\newcommand{\Dt}{\mathcal{D}(n)}

\DeclareMathOperator{\len}{length}
\DeclareMathOperator{\symp}{Sympl}

\begin{document}

\title[Moduli spaces of Delzant polytopes and symplectic manifolds]{Moduli spaces of Delzant polytopes and symplectic toric manifolds}

\author{\'Alvaro Pelayo \,\,\,\,\,\,\,\, Francisco Santos}

\address{\'Alvaro Pelayo,
Facultad de Ciencias Matem\'aticas,
Universidad Complutense de Madrid, 28040 Madrid, Spain and Real Academia de Ciencias Exactas, F\'isicas y Naturales, Spain.}
\email{alvpel01@ucm.es}

\address{Francisco Santos,
Departamento de Matem\'{a}ticas, Estad\'{i}stica y Computaci\'{o}n, Universidad de Cantabria, Av.~de Los Castros 48, 39005 Santander, Spain}
\email{francisco.santos@unican.es}



\begin{abstract}
This paper introduces modern geometric combinatorial technology from the theory of triangulations in order to derive results in toric symplectic geometry. 
In the main part of the paper we prove a number of properties of the space $\mathcal{D}(n)$ of $n$\--dimensional Delzant polytopes. 
Two  highlights are the construction of examples showing that, 
 in contrast with the classical work of Oda in dimension $2$, no classification of combinatorially minimal Delzant polytopes can be expected in dimension $3$ or higher, and a proof that the space of $n$\--dimensional 
 Delzant polytopes is path-connected. 
Our proof of the latter is based on the fact that every rational fan can be refined to a unimodular fan, which is a standard technique used for resolution of singularities of toric varieties.
In the last part of the paper, using the Delzant correspondence, these results allow us to answer several open questions concerning the moduli space $\mathcal{M}(n)$
of symplectic toric manifolds of dimension $2n$, since this space is isometric to the space of Delzant polytopes.
Our results imply that no classification of minimal models of symplectic toric manifolds is plausible in dimension $6$ or higher, which answers in the negative a long-standing folklore question originating in Oda's work (1978).
\end{abstract}

\maketitle
\setcounter{tocdepth}{2}
\tableofcontents
\section{Introduction}

\subsection{Novelty and achievements}

 The novelty in this paper  is the introduction of modern geometric combinatorial and convex geometric ideas into the study of Delzant polytopes and the moduli spaces they form, which we then incorporate into the framework of symplectic geometry. 
 These techniques  include using Hausdorff distance and its close relation to Minkowski sums,
focusing 
on the normal fans of Delzant polytopes rather than on the polytopes themselves, and applying to the normal fans ideas coming from the theory of triangulations of polytopes and vector configurations developed in the past thirty years (see for example the book~\cite{triangbook}) but which had not made their way into mainstream symplectic geometry yet. 

A \emph{Delzant polytope} in $\R^n$  is a simple full-dimensional polytope with rational edge directions and such that the primitive edge-direction vectors at each vertex form a basis of the lattice $\Z^n$. Equivalently, it is a polytope with a simplicial and unimodular normal fan (Definition~\ref{delpol}). We refer to Figure \ref{fig:delzants} for some examples.
 \begin{figure}[htb]
  \includegraphics[height=4.5cm]{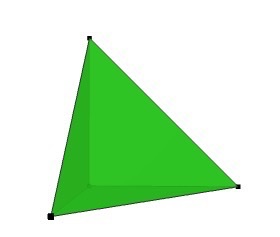} \qquad
 \includegraphics[height=4.5cm]{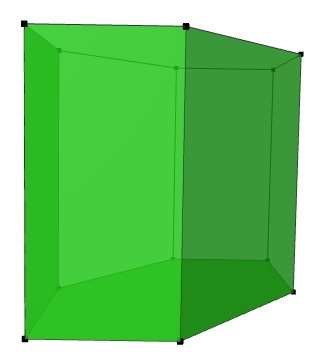}
\caption{Two Delzant polytopes in dimension three. Under the Delzant correspondence, the first one (tetrahedron) corresponds to the toric varitety $\C P^3$.}
\label{fig:delzants}
\end{figure}

Our two main achievements are:
\begin{itemize}
\item  We exploit the interplay between fans and polytopes (Figure~\ref{fig:Hirzebruch}), and use triangulation theory and resolution of singularities of toric varieties in order to construct examples which show that the famous classification of Delzant polygons by Oda~\cite[Theorem 8.2]{OdaMiyake} cannot have an analogue in dimension three or higher.
\item We bring techniques from convex geometry to gain a rather complete understanding of the geometry and topology of  the moduli space of $n$\--dimensional Delzant polytopes, extending to arbitrary dimension the results obtained for $n=2$ by Pelayo, Pires, Ratiu and Sabatini~\cite[Theorem 2]{PPRS14}.  
\end{itemize}

While the core ideas of the paper come from geometric combinatorics and convex geometry, they are strongly related to and motivated by developments in symplectic geometry, which build on seminal work of Atiyah, Delzant, Guillemin, Duistermaat, Heckman, Kostant, Sternberg, Weinstein, and others; see \cite{PES} for a survey of results in this direction.
In the last section of this paper we discuss this connection and the strong implications of the achievements mentioned above in toric symplectic geometry.

 \begin{figure}[htb]
 \includegraphics[height=3cm]{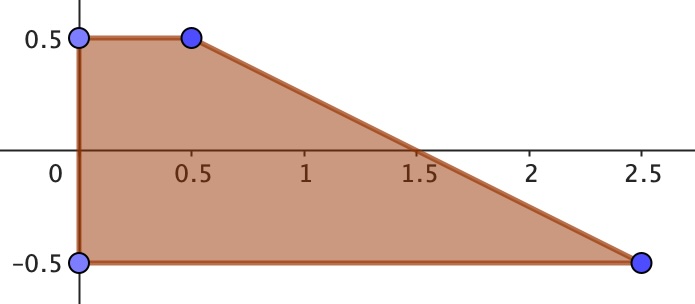} \qquad
 \includegraphics[height=4.5cm]{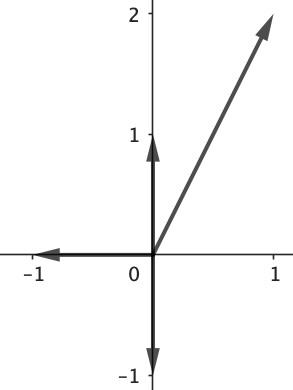}
\caption{The Hirzebruch trapezoid ${\rm H}_{3/2,1,2}$ with vertices $(0, -\tfrac12)$, $(0, \tfrac12)$, $(\tfrac52,-\tfrac12)$, $(\tfrac12,\tfrac12)$  and its corresponding normal fan, with generators $(-1,0)$, $(0,1)$, $(0,-1)$, and $(1,2)$. See Corollary~\ref{coro:2dim} for the meaning of its parameters.}
\label{fig:Hirzebruch}
\end{figure}

 \subsection{From discrete to symplectic geometry and back}

Our combinatorial results have a direct translation into the world of symplectic toric manifolds. 
As detailed in Definition~\ref{def:symplectic} these are quadruples of the form $(M, \om, \T^n, \mu)$ where $\T^n$ is an $n$\--dimensional torus acting effectively and Hamiltonianly on 
 a $2n$\--dimensional compact connected symplectic manifold $(M,\omega)$
 with momentum map 
 \[
 \mu=(\mu_1,\dots,\mu_n): M \to \mathfrak t^* \cong \R^n,
 \]
 where $\mathfrak t^*$ is the dual Lie algebra of  $\T^n$ and there is a choice of isomorphism $\mathfrak t^* \cong \R^n$ (see Section~\ref{stm} for details and Figure~\ref{fig:S2} for the simplest example).
 \begin{figure}[htb]
 \includegraphics[height=4cm]{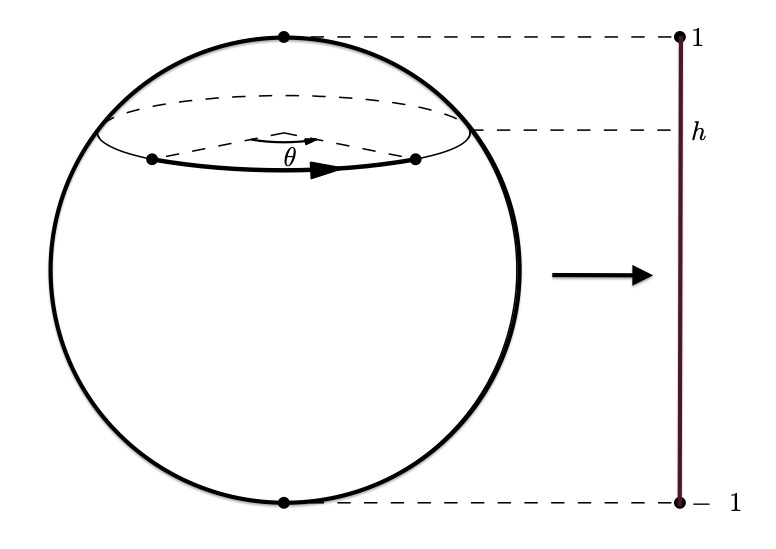}
 \caption{The only $2$-dimensional symplectic toric manifold, up to scaling, $(\C P^1 \cong \S^2, \omega = {\rm d}\theta \wedge {\rm d}h, \T^1=\S^1, \mu(\theta, h) =h)$. The momentum polytope $P=[-1,1] \subset \mathfrak{t}^* \cong \R$  classifies  its $\S^1$-equivariant symplectomorphism class. The Delzant correspondence sends the class of $(\S^2, \omega, \S^1, \mu)$ to $P$.}
 \label{fig:S2}
 \end{figure}

Often, symplectic toric manifolds are referred to as
\emph{toric integrable systems}, since $\mu$ is in fact a very special type of finite dimensional integrable Hamiltonian
system: one for which all components $\mu_i \colon M \to \R$ generate periodic flows of the same period.

The translation of results from the theory of Delzant polytopes to toric symplectic geometry
 is made via the \emph{Delzant correspondence} (Theorem~\ref{thm:delzant}), which is
a seminal result of Delzant~\cite{De1988} saying that the map
\begin{align*}
\Mo & \longrightarrow \Dt\\
[(M, \om, \T^n, \mu)] &\longmapsto \mu(M)
\end{align*}
is a bijection from the set of isomorphism classes of  compact connected symplectic toric manifolds of dimension $2n$, which we denote by $\Mo$,
to the set of Delzant polytopes of dimension $n$, which we denote by $\Dt$. Here, two such manifolds are \emph{isomorphic} if there is a toric equivariant symplectomorphism between them preserving  the momentum map (Section~\ref{isomorphism}).
The Delzant correspondence establishes a powerful bridge between the ``world of polytopes'' and the ``world of manifolds'', which is a very rare occurrence in differential geometry. It builds on seminal work of Atiyah~\cite{At82}, Guillemin-Sternberg~\cite{GS82}, Horn~\cite{HORN},  Kostant~\cite{Kostant}, and Schur~\cite{SCHUR}.

The use of the Delzant correspondence to work with polytopes instead of manifolds appeared 
in symplectic geometry around 20 years ago 
in the papers \cite{Pe06,PEE,McDuffTolman} in a more elementary, from the combinatorial point of view, form.  
(See also the recent paper \cite{AHN} for a different approach using convex geometry techniques in symplectic geometry).
The definition and moduli space structure on $\Mo$  was introduced in~\cite{PPRS14},  and has since then been considered by other authors~\cite{FuOh,FuKiMi}. It is constructed by pulling back to $\Mo$ the symmetric difference distance on $\Dt$.  The metric on $\Mo$ is closely related to the famous Duistermaat-Heckman measure~\cite{DH}. 

The metrics on $\Mo$ and $\Dt$ obtained in this way have the following nice properties:

\begin{itemize}
\item The symmetric difference distance and the Hausdorff distance induce the same topology on $\Dt$. This allows us to use the latter, which is better-understood in convex geometry, to deduce results about the former, hence about the Duistermaat-Heckman measure.
\item The diffeomorphism strata in $\Mo$ correspond to the normal equivalence strata in $\Dt$. Within such strata, continuity 
with respect to the distance of polytopes is the same as continuity regarding parallel translation of their facets.
\item The closures of the strata of $\Mo$ correspond to the symplectic toric varieties obtained by exploding lower dimensional invariant submanifolds, which includes blow-ups, but also more complicated explosions.
\end{itemize}

Having these moduli
space structures is useful to study the variation of properties such as continuity of functions defined on these moduli spaces, 
see Figalli\--Pelayo~\cite{FIPAPE} and Figalli\--Palmer\--Pelayo~\cite{FIPE}.
We note that other moduli spaces of polygons have been studied by a number of authors, for example 
see Hausmann\--Knutson~\cite{HaKn1, HaKn2} and Kapovich\--Millson~\cite{KaMi}.

\subsection{Main results}

Next we describe the main contributions of the paper to geometric combinatorics, convex geometry and symplectic geometry. Sections~\ref{sec:dp} and~\ref{sec:dp2} do not require any knowledge of symplectic geometry and the results therein are of interest independently of their connection to symplectic geometry, which is worked out in Section~\ref{sec:symplectic}.

\subsubsection{Minimal models of Delzant polytopes and symplectic toric manifolds}
\label{subsec:minimal}
In dimension $2$  there is a well-known classification of the minimal unimodular complete fans ---equivalently, of Delzant polygons--- going back to Oda (1978)~\cite[Theorem 8.2]{OdaMiyake}, which we state as Lemma~\ref{lemma:2dim}.
This  is  the main tool used in~\cite{PPRS14} to study the space $\mathcal{D}(2)$. We show that the failure of such a classification in dimension three and higher is drastic: there are Delzant polytopes with arbitrarily many vertices and facets and which do not admit any blow-down (Theorem~\ref{thm:dim3})  or, more strongly, which are not refinements of other unimodular fans (Theorem~\ref{thm:isolated}).

Oda's Classification implies a classification of $4$-dimensioanl symplectic toric manifolds in terms of finite sequences of blow-ups of three basic minimal models: $\C P^2$, $\C P^1\times \C P^1$, and the Hirzebruch surfaces. 
We state the symplectic versions of Theorems~\ref{thm:dim3} and Theorem~\ref{thm:isolated} as 
Theorems~\ref{thm:dim3-symplectic} and Theorem~\ref{thm:isolated-symplectic}. This answers in the negative the long-standing folklore question in toric geometry  of whether an analogue of Oda's classification of symplectic toric manifolds  exists in dimension higher than $4$.

\subsubsection{Local topology, compactness, and path-connectedness of $\Dt$ and $\Mo$}
Indeed, 
Theorem~\ref{thm:connected} shows that the moduli space $\Dt$ of $n$\--dimensional 
Delzant polytopes is path-connected but not locally compact, and describes its completion under the two natural metrics on it.  
Theorem~\ref{thm:quotientD} states the same for the quotient of $\Dt$ under unimodular equivalence, which we denote by $\Dtt$,
solving in particular a problem posed by  Fujita and Ohashi~\cite[Problem~5.2]{FuOh} in 2018.

Our proof of path-connectedness is algorithmic. It is based on the fact that every pair of unimodular fans admit a common unimodular refinement  (Corollary~\ref{coro:desingular}), which can be computed via  the standard resolution of singularities in toric varieties. Yet, contrary to the $2$-dimensional case, in dimension three and higher there is no  a priori bound on the complexity of the resulting 
path (Section~\ref{sec:pm}).

As said above, the Delzant correspondence sends $\Dt$ bijectively to the space $\Mo$ of $2n$-dimensional symplectic toric manifolds modulo isomorphisms. The quotient $\Dtt$ by unimodular equivalence then corresponds to quotienting $\Mo$ by \emph{weak isomorphsims}; we denote this quotient space by $\Mt$. With this in mind, 
Theorems~\ref{thm:connected} and \ref{thm:quotientD} translate to the symplectic statement Theorem~\ref{key}, which 
essentially solves  \cite[Problem 2.42]{PeVN2012}, also posed 
as \cite[Problem 3]{PPRS14}. In fact, the path-connectedness part extends the main
theorem of~\cite{PPRS14}  from $n=2$ to arbitrary $n\geq 2$.  

\subsubsection{Simple connectedness of $\mathcal{D}(2)$ and $\mathcal{M}(4)$}
We give partial results on the structure of the fundamental group of $\Dt$, showing that loops with a certain finiteness condition are null-homotopic (Corollary~\ref{coro:finite} and Theorem~\ref{thm:refining}). For $n=2$ all loops have the required property, so  $\DD(2)$ is simply connected (Corollary~\ref{coro:simply2}). For $n\geq 3$ loops without the required property exist (Example~\ref{exm:nonlocal}),  and whether the fundamental group of $\mathcal{D}(n)$ is trivial or not
remains an open question (Question~\ref{q:simply-connected}).

The symplectic version of Corollary~\ref{coro:simply2} ($\MM(4)$ is simply-connected) is Proposition~\ref{pol}.

\subsubsection{Stratification of $\Dt$ and $\Mo$}
 We study the stratification of $\Dt$ by normal fans.
 Let $\NN$ be a unimodular fan of dimension $n$ with $m$ rays.
 We show that the set of Delzant polytopes with normal fan equal to $\NN$ form a subspace of 
$\Dt$ isometric to a  polyhedral relatively open cone of dimension $m$ with a lineality space of dimension $n$.
The faces of the cone correspond to the fans that are refined by $\NN$ (Corollary~\ref{coro:cells}). 

This in particular allows us to give the exact value of the Hausdorff dimension of the space of Delzant $n$-polytopes with a fixed number of vertices or of facets~(Theorem~\ref{thm:dimension}), thus solving 
two problems  of  Fujita and Ohashi~\cite[Problems~5.3 and 5.5]{FuOh}.

\subsubsection{{\rm CW} topology on $\Dt$ and $\Mo$}
The stratification of the previous paragraph suggests that one can consider $\Dt$, hence also $\Mo$, as embedded into a {\rm CW} complex. The {\rm CW} topology obtained in this way is strictly finer than the metric topology (Propositions~\ref{prop:CW} and \ref{prop:finer}) and turns out to be more manageable from a homotopical point of view: with this topology $\Dt$ and $\Mo$ are \emph{weakly contractible} (Corollary \ref{coro:homotopy}).

\subsection{Structure of the paper}

In order to make the paper accesible for both the combinatorics and symplectic communities we have reviewed the basic concepts of both areas that are needed for the paper. An introduction to polytope theory in general and Delzant polytopes in particular is contained in Sections~\ref{sec:polytopes} and~\ref{sec:refinement} and an introduction to toric symplectic geometry and momentum maps is given  in Section~\ref{sec:symplectic-intro}. The rest of Section~\ref{sec:dp} contains the constructions of $3$-dimensional Delzant polytopes that lead to the results related to the impossibility of an Oda type classification, as mentioned in Section~\ref{subsec:minimal} above.

Section~\ref{sec:dp2} deals with the geometry and topology  of the space of Delzant polytopes. 
After introducing some technical tools such as the Hausdorff metric and Minkowski paths in Section~\ref{sec:metric}, in Section~\ref{sec:connected} we show that the space of Delzant $n$-polytopes is path-connected. 
Section~\ref{sec:stratification} studies the stratification of $\Dt$ by normal fans, and Section~\ref{sec:simply} shows that $\DD(2)$ is simply connected. In Section~\ref{sec:CW-contractible} we use the aforementioned stratification to endow $\Dt$ with a CW topology, and show that this new topology is weakly contractible.

After reviewing the basics of toric symplectic geometry in Section~\ref{sec:symplectic-intro}, in  Section~\ref{sec:space-symplectic} we interpret our combinatorial results in this world.

\subsection*{Acknowledgements}

The first author is funded by a BBVA (Bank Bilbao Vizcaya Argentaria) Foundation Grant for Scientific Research Projects with project title \emph{From Integrability to Randomness in Symplectic and Quantum Geometry}. 
The  second author is funded by grant PID2019-106188GB-I00 and PID2022-137283NB-C21 of MCIN/AEI/ 10.13039/501100011033 and by project CLaPPo (21.SI03.64658) of Universidad de Canta\-bria and Banco Santander.

We are grateful to the Department
of Mathematics, Statistics and Computation at the University of Cantabria for inviting the first author in October 2022 for a visit during which the first ideas of this paper were born. The first author is also thankful to Carlos Beltr\'an and Fernando Etayo for the hospitality during his visit.

We thank Liat Kessler and Margaret Symington for pointing out to us an error in a previous version of the proof of Theorem~\ref{thm:isolated}, and Ignasi Mundet for a suggestion leading to the statement of Corollary~\ref{coro:homotopy}, which is more general than the statement we previously had.
We also thank an anonymous referee for a very attentive and thorough reading and many useful comments which have helped us improved the paper.

\section{Delzant polytopes: basic theory and impossibility of an Oda type classification} 
\label{sec:dp}

In this section we review some modern tools of geometric combinatorics and convex geometry which we need to prove our main results. We also include several new examples
and counterexamples with interesting properties (notably, Theorem~\ref{thm:dim3} and Theorem~\ref{thm:isolated}) which illustrate striking differences between the geometry/topology of the space of  Delzant polygons and its higher dimensional analogue. 
In particular, these examples tell us that no Oda type classification of Delzant $n$-polytopes is possible in $n\ge 3$.

Throughout the section we emphasize the interplay between fans (more
commonly used in algebraic geometry) and polytopes (more commonly used in symplectic geometry), using one or the other depending on how convenient
it is for the proofs.

\subsection{Polytopes and their normal fans}
\label{sec:polytopes}

Because of the close relation which exists between Delzant polytopes and smooth toric varieties, we mainly use the books \cite{CLS, Ewald, Fulton} for background in polytope theory. Occasionally we use \cite{triangbook,Ziegler} for combinatorial aspects, and \cite{Gruber} for convex geometric ones.

\subsubsection{Polytopes}

Let $V$ be a real vector space of any finite dimension. Let $V^\ast$ be its dual vector space  and 
let 
\[
\left\langle\,\ ,\  \right\rangle:V ^\ast\times V \rightarrow\mathbb{R}
\]
 be the standard pairing. 

\begin{definition}
 A  \emph{polytope} $P$ in $V$ is the closed
convex hull of a finite subset $\{p_1, \ldots, p_m\} \subset V$, i.e., 
the  smallest convex set containing it or, equivalently,
\[
\Conv\{p_1, \ldots, p_m\} : = 
\Big\{\sum_{i=1}^m \lambda_i p_i \,\Big|\, \lambda_i \in [0,1],
\;\sum_{i=1}^m \lambda_i = 1 \Big\}.
\]
\end{definition}

Many authors call ``convex polytope'' what we call a polytope. Since we are not concerned with non-convex polytopes we omit the word ``convex'', as done for instance in \cite{triangbook, Ziegler}.

The \textit{dimension} of a convex set $S\subset V$ is the dimension of its affine span.
The set is \textit{full dimensional} if it affinely spans $V$.
(For example, Figure~\ref{fig:delzants} shows two full-dimensional polytopes in $V =\R^3$.)

By definition, a  polytope is a compact subset of $V$.  
An \emph{extreme point} of a convex subset $C\subset V$ is a point of $C$ which is not 
in any open line segment  contained in $C$. A
 polytope is equal to the closed convex hull of its extreme 
points (Krein\--Milman Theorem \cite{KeMi1940}), and 
this is the unique minimal description of a polytope as a convex hull of points.%
\footnote{For a polytope, extreme points are the same as \emph{vertices} (that is, $0$-dimensional faces as defined below).
For general convex sets, every vertex is an extreme point but the converse is not always true.} 
Polytopes have also a dual description via inequalities, that is, as finite intersections of half-spaces.

\begin{definition}
A \emph{half-space} in $V$ is any subset which has the form
\[
\Big\{p\in V \,\Big|\,\left\langle \alpha,p\right\rangle \leq b\Big\},
\]
where $\alpha \in V^*\setminus \{0\}$ and $b\in \R$.
 We call a  \emph{polyhedron}  any finite intersection of half-spaces.
 \end{definition}

By the Minkowski-Weyl Theorem \protect{\cite[Theorem 14.2]{Gruber}
every polytope is a  polyhedron and every bounded  polyhedron is a polytope.

The language of inequalities allows us to define faces of a polytope, understanding the vector $\alpha$ in each inequality as a normal vector to the corresponding face:

\begin{definition}
Let $P$ be a polytope, or a  polyhedron. For each linear functional $\alpha\in V^*$ the \emph{face of $P$ in the direction of $\alpha$} is
\[
{\rm F}_\alpha(P):=\Big\{p \in P \,\Big|\,\left\langle
\alpha,p-x\right\rangle \geq 0 \text{ for } x\in P\Big\}.
\]
\end{definition}

Every face of a polyhedron (respectively, of a polytope) is itself polyhedron (respectively, a polytope), of a certain dimension. 
The set of all faces is a partially ordered set (\emph{poset}) ranked by dimension. It has a unique maximal element, namely $P$ itself which equals ${\rm F}_0(P)$. 
Faces of dimension 0 are called \emph{vertices} of $P$. Faces of dimension one are called \emph{edges} 
if they are bounded, which is always the case if $P$ itself is bounded.
Faces of dimension $\dim(P)-1$ are called  \emph{facets}.%
\footnote{Some authors consider the empty set as a face of $P$ too (of dimension $-1$), but we do not need that.}

A  \emph{cone} in $V$ is any subset that is closed under addition and multiplication by nonnegative scalars. 
A \emph{polyhedral cone} is a  cone that is also a polyhedron. Equivalently, it is a finitely generated cone, according to the following formula for the \emph{cone generated by a finite set} $v_1,\dots,v_m\in V$ of vectors:
\[
\Cone\big\{v_1,\dots,v_m\big\} : = 
\Big\{\sum_{i=1}^m a_i v_i\;\Big|\; a_i \in [0,\infty) \Big\}.
\]
A polyhedral cone is \emph{simplicial} if it is generated by linearly independent vectors.

The \emph{dual} or \emph{normal cone} of a cone $C\subset V$ is defined as
\[
C^\vee :=\Big\{ \alpha \in V^* \,\Big|\, \langle \alpha, x\rangle \leq 0 \text{ for } x \in C\Big\} \subset V^*.
\]
It is easy to show that $(C^\vee)^\vee=C$ for every closed  cone, and that the dual of a  polyhedral cone is a polyhedral cone.

To every face $F$ of  a polytope (or polyhedron) $P$ one naturally associates two closed  polyhedral cones, living respectively in $V$ and $V^*$, and
dual to one another:
the \emph{tangent cone} and the \emph{normal cone} of a face $F$ of $P$ are the  polyhedral cones 
\begin{align*}
{\rm T}(P,F):=& \R_{\geq 0} (P-F) =  \Big\{\lambda (q-p) \,\Big|\, q\in P, \, p\in F,\, \lambda \geq 0  \Big\} \subset V,\\
{\rm N}(P,F):=&  \Big\{\alpha\in V^* \,\Big|\, F \subseteq {\rm F}_\alpha(P)\Big\} = \Big\{\alpha\in V^* \,\Big|\,  \left\langle\alpha,p-q\right\rangle \geq 0 \ \text{ for } p\in F, q\in P\Big\} \subset V^*.
\end{align*}
See Figure~\ref{fig:cones} for an example.

\begin{figure}[htb]
 \includegraphics[height=3.7cm]{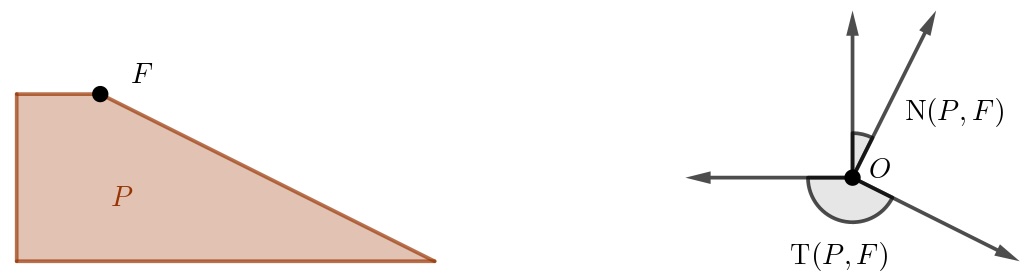} 
\caption{Left: a polytope $P$ (a quadrilateral) and a face $F$ of it (a vertex). Right: the tangent and normal cones of $F$ in $P$.}
\label{fig:cones}
\end{figure}

\begin{definition}
A \emph{fan} in $V$ is a complex of polyhedral cones. That is, a collection $\NN$ of polyhedral cones in $V$ that is closed under taking faces and such that if $C_1$ and $C_2$ are in $\NN$ then $C_1\cap C_2$ is a common face of both. 
A fan is called \emph{simplicial} if all its cones are simplicial.

The  \emph{normal fan} $\NN(P)$ of a polytope $P$, is the fan consisting  of the normal cones of all faces of $P$. See Figure~\ref{fig:Hirzebruch} for an Example.
\end{definition}

The  cones of $\NN(P)$ form again a poset, ranked by dimension, and this poset is antiisomorphic (that is, isomorphic with the order reversed) to the face poset of $P$. 
For example, $\NN(P)$ has a unique minimal cone, the normal cone of $P$ itself, which equals $\{0\}$ if and only if $P$ is full-dimensional. 
When this happens  1-dimensional cones of $P$ are \emph{rays}, that is, vector half-lines, and they are in bijection to facets of $P$: each ray in $\NN(P)$ is generated by a normal vector to the corresponding facet. In this case $P$ has a unique minimal expression as an intersection of half-spaces, namely the one that uses the half-spaces containing facets of $P$ in their boundary.

The following are some properties that a polyhedron $P$ may or may not have, and the corresponding properties for its normal fan $\NN(P)$. We assume $P$ to be full-dimensional and call $n$ the dimension of $P$ and $V$:
\begin{itemize}
\item $P$ is bounded, i.e., a polytope, if and only if its normal fan is \emph{complete}; that is, if every $\alpha\in V^*$ lies in the normal cone of some face.
\item $P$ is full-dimensional if and only if its normal fan is \emph{pointed}, that is, if the lineality space of $\NN(P)$ is the cone $\{0\}$. 
Here, the \emph{lineality space} of a cone $C$ is  
\[
\lin(C):= \Big\{x\in C \,\Big|\, -x \in C\Big\}.
\]
It equals the unique maximal linear space contained in $C$, as well as the unique minimal face of $C$. In a fan, all cones have the same lineality space. For example, for a polytope (or polyhedron) $P$
the lineality space of $\NN(P)$ equals the normal cone of $P$ considered as a face of itself, that is, the space of linear functionals that are constant on $P$. It is also the linear space orthogonal to the affine span of $P$.

\item A full-dimensional polytope (or polyhedron) $P$ is called \emph{simple} if every vertex lies in exactly $n$ facets. This happens if and only if $\NN(P)$ is simplicial.
\end{itemize}

In dimension two every complete fan is the normal fan of some polygon. In dimension three or higher the same is not true.  We say that a complete fan is \emph{polytopal} if it equals the normal fan of some polytope. Otherwise it is \emph{non-polytopal}.

\begin{example}[A non-polytopal icosahedral fan]
\label{example:cubocta}
Consider the cubeoctahedron, that is, the convex hull of the mid-points of the twelve edges of the cube $[-1,1]^3\subset \R^3$. 
(See, e.g., Example 3.6.17 and Figure 3.47 in \cite{triangbook}).
It has six square facets and eight triangular ones. If each square facet is triangulated by inserting one of its two diagonals, we get a triangulated sphere with 20 triangles. A particularly symmetric way of doing this is by choosing the six diagonals to use each of the twelve vertices once, producing a complex isomorphic to the icosahedron. 
It turns out that the simplicial fan obtained by coning from the origin to each of the 20 triangles is not polytopal, as shown, in a different language, in \cite[Example 3.6.17]{triangbook}. 
\end{example}

\begin{example}[The smallest non-polytopal fan]
\label{example:moae}
Consider the negative standard basis vectors $-{\rm e}_i$ and the vectors 
\[
a_i=({\rm e}_1+{\rm e}_2+{\rm e}_3)-{\rm e}_i,\,\,\,\, i=1,2,3.
\] 
Observe that their convex hull is a triangular prism. Triangulate the boundary of the prism inserting diagonals in the quadrilateral faces in a cyclic way: ${\rm e}_1a_2$, ${\rm e}_2a_3$ and ${\rm e}_3a_1$. Consider the fan obtained coning from the origin to this triangulation, which has eight cones and six rays. Combinatorially, but not metrically, this fan coincides with the normal fan of the standard 3-cube, that is, the central fan of the regular octahedron.
But geometrically this fan is not polytopal; it is closely related to the so-called ``mother of all examples'' featured prominently in \cite{triangbook}. 
Every complete fan of dimension $n$ with less than $n-3$ rays is polytopal (\cite[Section 5.5]{triangbook}), so this is the smallest possible non-polytopal fan.
\end{example}

There even exist complete simplicial fans that are not \emph{combinatorially equivalent} to the normal fan of any polytope.\footnote{In fact, these fans are the ones usually called \emph{non-polytopal} in geometric combinatorics. See, e.g., \cite[Section 9.5]{triangbook}. What we call polytopal fans are called \emph{coherent}, \emph{regular}, or \emph{projective} fans, depending on the author.} The smallest one has dimension four, eight rays and 19 maximal cones, and can be derived by coning over Barnette's sphere~\cite[Example 9.5.23]{triangbook}.

\subsubsection{Unimodular fans and Delzant polytopes}

From now on we assume that $V=\R^n$. 
A fan in $\R^n$ is called \emph{rational} if the linear span of every face is a rational subspace. When the fan is pointed (e.g., when it is the normal fan of a full-dimensional polytope) a necessary and sufficient condition for rationality is that every ray contains a non-zero lattice point.
For example, the normal fan of a full-dimensional polytope $P$ is rational if, and only if, all facet normals of $P$ have rational directions.

\begin{lemma}
\label{lemma:rational}
Let $P$ be a full-dimensional polytope.
 Then, all edge directions of $P$ are rational if and only if all facet-normal directions of $P$ are rational.
\end{lemma}

\begin{proof}
The direction of an edge $e$ is the unique orthogonal line to all the  normals of facets containing $e$. The direction normal to a facet $F$ is the unique orthogonal line to all the edges contained in $F$.
\end{proof}

Observe that every rational ray $R$ has a unique \emph{primitive generator}, the unique nonzero integer vector in $R$ with relatively prime entries. The same happens in a rational line, except that the primitive generator is now unique only up to sign.
We will often abuse language and say that a certain vector $v\ne 0$ ``is'' a ray in a fan, meaning that it spans a ray. When the ray is rational we will typically assume $v$ to be primitive.
The \emph{multiplicity} of a rational simplicial cone $C$ with primitive rays $v_1,\dots,v_m$ is the index of the sublattice generated by $\{v_1,\dots,v_m\}$ as a subgroup of the ambient lattice $\R C \cap \Z^n$. If $C$ is full-dimensional, and hence $m=n$, then the multiplicity of $C$ equals the absolute value of the determinant $\det(v_1,\dots,v_n)$.

\begin{definition}
A \emph{unimodular fan} in $\R^n$ is a rational simplicial full-dimensional fan in which every cone has multiplicity one.
\end{definition}

If the cone $C_2$ is a face of a cone $C_1$ then the multiplicity of $C_2$ divides that of $C_1$. Hence, in the definition of unimodular only the maximal cones need to be considered.

 \begin{proposition}
 \label{prop:smooth}
 For each vertex of a full-dimensional simple polytope (or polyhedron) $P$, the tangent cone is unimodular if and only if the normal cone is unimodular. 
 \end{proposition}

 \begin{proof}
Let $n$ be the dimension.
Since the tangent cone and the normal cone at a vertex are orthogonal to one another,  we can consider the lattice $\Lambda_v$ generated by the primitive edge vectors at a vertex $p$. This is a sublattice of $\Z^n$, of a certain index $K$, and its dual lattice $\Lambda^*_v$ is generated by the facet-normals at $v$, suitably normalized. Hence, $\Lambda^*_v$ is a superlattice of $\Z^n$ with the same index $K$. We then have 
$
\Lambda_v=\Z^n \text{ if and only if } \Lambda^*_v=\Z^n. 
$
That is, the primitive facet normals are a basis of $\Z^n$ if and only if the primitive edge-direction vectors are.
\end{proof}

The following class of polytopes owe their name to Delzant's paper~\cite{De1988} (see \cite[p. 8]{Guillemin}),
 in which he uncovers a deep and far-reaching relation between symplectic geometry and geometric combinatorics, which we will exploit in the last section of this paper.

\begin{definition} \label{delpol}
A \emph{Delzant polytope} of dimension $n$ is a full-dimensional polytope $P\subset \R^n$ whose normal fan is unimodular. Equivalently, it is a full-dimensional polytope satisfying:
\begin{itemize}
\item Its normal fan is \emph{rational}. Equivalently, its edge directions are rational;
\item It is \emph{simple}: each vertex is contained in exactly $n$ edges or, equivalently, in exactly $n$  edges;
 \item It is \emph{smooth}: all its cones have multiplicity one. Equivalently, for every vertex $v$, 
 the primitive vectors along the $n$ edges containing $v$ form a basis of the lattice $\Z^n$.  
 \end{itemize}
 \end{definition}
 
 Unimodular fans may be non-polytopal, and hence not be the  normal fans of any Delzant polytope. In fact, slight modifications of examples \ref{example:cubocta} and \ref{example:moae} show this:

 \begin{itemize}
\item In Example~\ref{example:cubocta} all the full-dimensional cones have determinant equal to $\pm 2$, but their generators lie in the sublattice 
\[
\Big\{(x,y,z)\in \Z^2 \,\Big|\, x+y+z\equiv 2\pmod 2\Big\},
\]
 of index two. Thus, a linear transformation sending this sublattice to $\Z^2$ produces a unimodular fan that is still not polytopal.
 In fact, if this unimodular fan was polytopal it would be the normal fan of a Delzant polytope combinatorially isomorphic to the  dodecahedron, which does not exist: by \cite[Theorem 3.1 and Lemma 3.3]{Delaunay} every Delzant $3$-polytope has at least one triangular or quadrilateral face.
 
 \item The fan in Example~\ref{example:moae} is almost unimodular: it has a unique non-unimodular full-dimensional  cone, namely the cone of $(a_1,a_2,a_3)$ which has determinant two. But this cone can be easily divided into three unimodular cones by a stellar subdivision (see definition below) inserting the vector ${\rm e}_1+{\rm e}_2+{\rm e}_3 = \frac12(a_1,a_2,a_3)$. This gives a unimodular but not polytopal fan, with seven rays and 10 cones.
\end{itemize}

However, by definition the normal fans of Delzant polytopes are exactly the complete unimodular \emph{and polytopal} fans. For this reason most combinatorial properties of Delzant polytopes can be stated, and are often easier to deal with, via their normal fans.

\begin{remark}
Unimodular fans modulo $\GL(n,\Z)$ equivalence are in bijection to smooth 
toric varieties modulo isomorphism. 
More generally,  rational full-dimensio\-nal fans biject to normal toric varieties.
In this bijection simplicial fans correspond to varieties that have only finite quotient singularities, complete  fans correspond to complete varieties, and polytopal fans to projective ones. 
Faces of a fan are in bijection with (algebraic) torus orbits, and the multiplicity of each quotient singularity (or orbit of them) equals the multiplicity of the corresponding simplicial cone. 
\end{remark}

\subsection{Stellar subdivisions, desingularization and $2$-dimensional classification}
\label{sec:refinement}

\subsubsection{Stellar subdivision and desingularization}

Let $\Sigma$ and $\Sigma'$ be two complete fans in $\R^n$. We say that $\Sigma'$ refines $\Sigma$ if every cone $C'\in \Sigma'$ is contained in some cone $C\in \Sigma$. The fact that cones in a fan intersect properly  then implies that every cone of $\Sigma$ equals a union of cones of $\Sigma'$. 

A particularly simple and useful method to refine fans is by introducing a single new ray, as follows.

\begin{definition}[Stellar subdivision, \protect{\cite[Section 11.1]{CLS}, \cite[Definition 2.1]{Ewald}}]
\label{defi:stellar}
Let $\Sigma$ be a fan in $\R^n$ and let $\gamma\in V$ be a vector contained in some cone of $\Sigma$.
The  \emph{stellar subdivision} of $\Sigma$ by $\gamma$ is the fan consisting of the cones of $\Sigma$ not containing $\gamma$ together with the cones of the form $\Cone\{\gamma\cup C\}$ such that  $\gamma \not \in C\in \Sigma$ and there is a $C'\in \Sigma$ with $C\cup \{\gamma\}\subset C'$.
\end{definition}

Differently put, if we call \emph{star of $\gamma$ in $\Sigma$} the subfan of $\Sigma$ consisting of cones containing $\gamma$, what the stellar subdivision does is to introduce in $\Sigma$ (the ray generated by) $\gamma$ by centrally subdividing all cones in the star of $\gamma$.

In the case of a simplicial fan, which is our main interest, a more explicit definition of stellar subdivision is possible. Let $C=\Cone\{\alpha_1,\dots,\alpha_k\}$ be the unique minimal cone of $\Sigma$ containing $\gamma$, with $\alpha_1,\dots,\alpha_k$ being ray generators.
That is, $C$ is the cone containing $\gamma$ in its relative interior. Then, every cone containing $\gamma$ is of the form
\[
D=\Cone\big\{\alpha_1,\dots,\alpha_k, \beta_1,\dots,\beta_l\big\},
\]
for some, perhaps none, generators $\beta_1,\dots,\beta_l$. The stellar subdivision by $\gamma$ removes each such cone $D$ and inserts instead the $k$ cones
\[
D_i:=\Cone\left\{\big\{\gamma,\alpha_1,\dots,\alpha_k, \beta_1,\dots,\beta_{l}\big\} \setminus \big\{\alpha_i\big\}\right\}.
\]

The following properties of stellar subdivisions are easy to prove and left to the reader.

\begin{proposition}
\label{prop:stellar}
Let $\Sigma$ be a complete fan in $\R^n$ and let 
$\Sigma'$ be obtained from $\Sigma$ by performing a stellar subdivision at $\gamma$.
\begin{itemize}
\item[(1)] If $\Sigma$ is polytopal, then $\Sigma'$ is polytopal.
\item[(2)] If $\Sigma$ is simplicial, then $\Sigma'$ is simplicial.
\item[(3)] If $\Sigma$ is complete, then $\Sigma'$ is complete.
\item[(4)] If $\Sigma$ is rational and $\gamma$ is rational, then $\Sigma'$ is rational.
\item[(5)] If $\Sigma$ is unimodular and (with the notation above) the $\alpha_i$ are primitive and $\gamma= \sum_{i=1}^k \alpha_i$, then $\Sigma'$ is unimodular.
\end{itemize}
\end{proposition}

The unimodular case is particularly important, so we give it a name:

\begin{definition}
\label{defi:blowup}
Let $\Sigma$ be a fan in $\R^n$ and let $C\in \Sigma$ be a unimodular simplicial cone of $\Sigma$ with primitive generators $\alpha_1,\dots,\alpha_k$. The \emph{blow-up of $\Sigma$ at $C$} is the stellar subdivision by the vector $\gamma=\sum_{i=1}^k \alpha_i$.
We also say that $\Sigma$ is obtained from $\Sigma'$ by a \emph{blow-down}.
\end{definition}

The reason why we call this a \emph{blow-up} is that the toric variety of the fan $\Sigma'$ is the blow-up of the variety of $\Sigma$ at the torus-invariant orbit corresponding to the cone $C$. (See \cite[Definition 3.3.14 and Proposition 3.3.15]{CLS}, or \cite[Definition 7.1]{Ewald}).

\begin{theorem}[\protect{\cite[Theorem 11.1.9]{CLS}, \cite[Theorem 8.5]{Ewald}, \cite[Section 2.6]{Fulton}}]
\label{thm:desingular}
Let $\Sigma$ be any rational complete fan in $\R^n$. Then there is a sequence of stellar subdivisions producing a unimodular fan $\Sigma'$ that refines $\Sigma$. 
\end{theorem}

\begin{proof}
In a first step we refine $\Sigma$ to be simplicial. One systematic way of doing this is using the \emph{barycentric subdivision}, which can be expressed as a sequence of stellar subdivisions, one with a $\gamma$ chosen in the relative interior of each cone of any dimension of $\Sigma$, and performed in decreasing order of dimension.

Once we have a simplicial and still rational fan, each cone has a well-defined multiplicity. As mentioned above, the multiplicity of a simplicial cone $C$ with primitive generators $\alpha_1,\dots,\alpha_k$ is the index of $\sum_{i=1}^k \alpha_i\Z$ as a sublattice of the lattice $\lin(C) \cap \Z^n$. 
This equals the number of lattice points in the half-open cube
\[
\sum_{i=1}^k [0,1) \alpha_i \subset C.
\]
Hence, if the multiplicity of $C$ is greater than one, there is a $\gamma \in \Z^n\setminus \{0\}$ which can be expressed as
\[
\gamma = \sum_{i=1}^k \lambda_i \alpha_i,
\]
with $\lambda_i\in [0,1)$ for every $i$.
Under these conditions, a stellar subdivision by $\gamma$ lowers the multiplicity of all the cones involved. More precisely, with the notation after Definition~\ref{defi:stellar}, each cone $D$ of a certain multiplicity $m$  is substituted by $k$ cones $D_i$ of respective multiplicities $\lambda_im < m$, for $i\in\{1,\dots,k\}$. Of course, the number of cones increases, but iterating as follows we can make all the resulting cones unimodular: if we call $M>1$ the maximum multiplicity among the cones of $\Sigma$, performing a stellar subdivision with $\gamma$ chosen in the half-open cube of a cone of  multiplicity $M$ we make  the number of cones of that multiplicity decrease, and do not introduce cones of higher multiplicity. Hence, eventually we get a fan with maximum multiplicity $M' <M$.  Iterating this process, we eventually get a unimodular fan. 
\end{proof}

Observe that by Proposition~\ref{prop:stellar}, if $\Sigma$ is polytopal then $\Sigma'$ is polytopal too.
We often use the following immediate consequence of this theorem.

\begin{corollary}
\label{coro:desingular}
Let $\Sigma_1,\dots,\Sigma_k$ be rational complete fans in $\R^n$. Then, there is a unimodular fan $\Sigma'$ that refines every $\Sigma_i$. 
\end{corollary}

\subsubsection{Corner chopping and blow-ups}

In what follows two subsets $X,Y$ of $\R^n$ are called \emph{$\AGL(n,\Z)$\--congruent} if there exists
$c \in \R^n$ and $A \in \GL(n,\Z)$ such that $Y=A(X)+c$.

\begin{definition}
\label{defi:length}
The \emph{rational length} of a line segment $I \subset\R^n$ of rational slope is the unique number $\ell\in [0,\infty)$ such that $I$ is $\AGL(n,\Z)$-congruent
to an interval of length $\ell$ on a coordinate axis. It is denoted by $\len(I)$.
\end{definition}

\begin{definition} \label{cc}
Let $P$ be a Delzant polytope in $\R^n$. Let $v$ a vertex 
of $P$. Let 
\[
\Big\{v+t u_i \,\Big|\, 0\leq t\leq \ell_i\}\Big\}
\]
be the set consisting of edges emanating from $v$, where the vectors
$u_1,\ldots,u_n$ generate  $\Z^n$ and
$\ell_i=\len(u_i)$ for every $i \in \{1,\ldots, n\}$. Let  $\varepsilon>0$ be smaller than all the $\ell_i$'s. 
We define the \emph{corner chopping of size} $\varepsilon$ of $P$ 
at $v$ is the polytope $P'$ obtained from $P$ by 
intersecting it with the half space  
\[
\left\{v+\sum_{i=1}^n t_i u_i \Bigg|  \sum_{i=1}^n t_i\geq 
\varepsilon\right\}.
\]
\end{definition}

Let us translate this operation into the normal fan.

\begin{proposition} \label{prop:cc}
Let $P$ be a Delzant polytope in $\R^n$. Let $v$ a vertex of $P$ and let $P'$ be obtained from $P$ by a \emph{corner chopping of size} $\varepsilon$ at $v$. Let $\alpha_1,\dots, \alpha_n$ be the primitive  normal vectors of $P$ at the facets containing $v$ and let $a_i=\langle \alpha_i, v\rangle $ be the corresponding right-hand sides for the facet inequalities. Then
\[
P' = P \cap \left\{x\in V \,\Bigg| \, \big \langle  \sum_{i=1}^n\alpha_i, x\big\rangle \geq \sum_{i=1}^n a_i + \varepsilon \right\}.
\]
\end{proposition}

\begin{proof}
With the notation of Definition~\ref{cc}, observe that the bases $\{u_i\}_i$ of $V$ and $\{\alpha_i\}_i$ of $V^*$ are dual to one another; that is, $\langle \alpha_i, u_j\rangle = \delta_{i,j}$. 

Since any vector $x\in V$ can be written uniquely as $x=v+ \sum_{i=1}^n t_i u_i$ we only need to check that the inequality $v+\sum_{i=1}^n t_i u_i \mid \sum_{i=1}^n t_i \geq  \varepsilon$ from the definition is equivalent to the inequality $\big \langle  \sum_{i=1}^n\alpha_i, x\big\rangle \geq \sum_{i=1}^n a_i + \varepsilon$ in the statement. This follows from
\[
\left \langle  \sum_{i=1}^n\alpha_i, x\right\rangle =  
\left\langle \sum_{i=1}^n  \alpha_i, v+ \sum_{j=1}^n t_ju_j\right\rangle =
\sum_{i=1}^n \left\langle  \alpha_i, v\right\rangle + \sum_{i=1}^n t_i
= \sum_{i=1}^n a_i + \sum_{i=1}^n t_i.
\qedhere
\]
\end{proof}

That is to say, the normal fan of $P'$ is constructed from that of $P$ by inserting the ray $\sum_{i=1}^n \alpha_i$ in 
$\NN(P, v) = \Cone\{\alpha_1,\dots,\alpha_n\}$.

\begin{corollary}
If a Delzant polytope $P'$ 
is obtained from another one $P$  by a corner chopping then its normal fan $\NN(P')$ is obtained from $\NN(P)$ by a  blow-up at a full-dimensional cone  $C\in\NN(P)$.
\end{corollary}

The converse is also true, in the following sense:

\begin{corollary}
Let $P'$ be a Delzant polytope in $\R^n$. If its normal fan  $\NN(P')$ is the blow-up of a certain unimodular fan $\Sigma$ at a full-dimensional cone then there is a Delzant polytope $P$ with $\NN(P)=\Sigma$ and of which $P'$ is a corner chopping.
\end{corollary}

\begin{proof}
Let $\alpha_1,\dots,\alpha_n$ and $\alpha_0=\sum_{i=1}^n \alpha_i$ be the rays of $\Sigma'$ involved in the blow-up, and let $b_0,b_1,\dots,b_n$ be the corresponding right-hand sides in the inequality description of $P'$. Observe that all the cones $C_i$ in the definition of blow-up automatically have the same determinant as the cone $C$ generated by $\alpha_1,\dots,\alpha_n$. In particular the blow-down fan $\Sigma$ is unimodular, since $\Sigma'$ is.

The facet of $P'$ corresponding to $\alpha_0$ is an $(n-1)$-simplex, since it has $n$ vertices, corresponding to the $n$ cones $C_i$ in the definition of blow-up. Removing from the inequality description of $P'$ the inequality $\langle  \alpha_0, x\rangle \geq b_0$ gives a polytope $P$ in which that facet disappears and the adjacent $n$ facets (corresponding to the rays $\alpha_1,\dots,\alpha_n$) are enlarged, producing a unique new vertex $v$, the intersection of the hyperplanes $\{\langle  \alpha_i, x\rangle = b_i\}$, $i \in \{1, \dots,n\}$. 
The normal cone of $v$, generated by $\{\alpha_1,\dots,\alpha_n\}$,  is unimodular as said above.
Hence, $P$ is Delzant and $P'$ is obtained from it as in Proposition \ref{prop:cc}, that is, by a corner chopping.
\end{proof}

\subsubsection{Classification of Delzant polygons}

A two-dimensional fan is uniquely determined by its set of rays, since each two-dimensional cone is bounded by two rays and the rays must appear in their cyclic order. This easily implies that all two-dimensional fans are polytopal: if $\alpha_1,\dots,\alpha_n$ are generators for the rays of the fan $\Sigma$ and we identify $\R^2$ and $(\R^2)^*$ via the standard scalar product, then the polytope defined by the linear inequalities $\langle \alpha_i, x\rangle \leq \|\alpha_i\|$ has $\Sigma$ as its normal fan.

Any unimodular two-dimensional fan can be obtained by the following recursive recipe. 

\begin{lemma} [\protect{Oda~\cite[Theorem 8.2]{OdaMiyake}. See also \cite[Theorem 10.4.3]{CLS},\cite[Theorem 6.6]{Ewald},\cite[Section 2.5 and Notes to Chapter 2]{Fulton}}]
\label{lemma:2dim}
Let $\Sigma$ be a complete unimodular fan in $\R^2$.
\begin{itemize}
\item[(1)] If $\Sigma$ has only three rays then it is $\GL(2,\mathbb{Z})$-equivalent to the fan of the projective plane, with rays $(1,0)$, $(0,1)$ and $(-1,-1)$. 
\item[(2)] If $\Sigma$ has  four rays then it is $\GL(2,\mathbb{Z})$-equivalent to the fan of the Hirzebruch surface $\mathbb H_k$,  with rays $(0,-1)$, $(0,1)$ $(-1,0)$, and $(1,k)$, for some $k\in \Z_{\geq 0}$.
\item[(3)] If $\Sigma$ has more than four rays then it is an iterated blow-up of some unimodular fan with four rays.
\end{itemize}
\end{lemma}

The following translation of Lemma~\ref{lemma:2dim} into Delzant polygons is formulated in~\cite[Lemma 2.16]{KKP}.

\begin{corollary} 
\label{coro:2dim}
Let $P$ be a Delzant polygon in $\R^2$. 
\begin{itemize}
\item[(1)] If $P$ has exactly three edges then there exists a unique $\lambda>0$ such that 
$P$ is $\AGL(2,\mathbb{Z})$\--congruent to the Delzant triangle 
$$ 
P_\lambda := \left\{ \left.(x_1,x_2) \in \mathbb{R}^2 
\ \right| \
 x_1 \geq 0 ,\; x_2 \geq 0 ,\,\,\, x_1 + x_2 \leq \lambda 
 \right\}.
$$
\item[(2)]
 If $P$ has exactly with $4$ edges, then there exists $k\in \Z_{\geq 0}$ and positive numbers $a , b >0$ with $2a > bk$ such that $P$ is 
$\AGL(2,\mathbb{Z})$\--congruent to a Delzant polygon obtained from the Hirzebruch trapezoid 
$$ 
{\rm H}_{a,b,k} := \left\{ (x_1,x_2)\in\mathbb{R}^2 \ \left| \
   -\frac{b}{2} \leq x_2 \leq \frac{b}{2}, \right.\,\,\, 0 \leq x_1 
   \leq a - kx_2 \right\}, 
$$
 by a sequence of $s$ corner choppings (as defined in Definition~\ref{cc}).
\item[(3)] If $P$ has $4+s$ edges for $s\geq 1$ then it is obtained from a Hirzebruch trapezoid by a sequence of $s$ corner choppings.
\end{itemize}
\end{corollary}

The parameters in ${\rm H}_{a,b,k}$ have the following interpretation: 
$b$ is the height of the trapezoid, 
$a$ is its average width, and the integer $k$ governs the difference in slopes between the two non-parallel edges (with the trapezoid being a rectangle if $k=0$).
Observe also that ${\rm H}_{a,b,1}$ is a corner chopping of (a translation of) the triangle $P_{a+\frac{b}2}$. See Figure~\ref{fig:Hirzebruch} for a picture of ${\rm H}_{3/2,1,2}$.

\begin{remark}[The maximum area of a Delzant polygon]

One interesting question is how big can the area of a Delzant polygon be in terms of its number $k$ of vertices (equivalently, edges) and its perimeter $\ell$. Here, the perimeter is the sum of lengths of edges, with length as in Definition~\ref{defi:length}.
The following results answer this question asymptotically.%
\footnote{The statements in \cite{9authors} assume the Delzant polytopes in question to have only integer vertices, but the proofs go through without this assumption.}
Then:
\begin{itemize}
\item[(1)] Every Delzant polygon with $k$ vertices and perimeter $\ell$ has area bounded above by $\phi^{2k} \ell^2$, where $\phi=\frac{1+\sqrt5}2$ is the golden ratio~\cite[Theorem 28]{9authors}.

\item[(2)] There are Delzant polygons with $k$ vertices and perimeter $\ell$ with area $\geq \phi^{2k/3} \ell^2$~\cite[Example 30]{9authors}.
\end{itemize}
A lower bound in terms of $k$ and $\ell$ alone makes no sense:  one can construct Delzant polygons 
with a pair of parallel edges and arbitrarily many additional edges. Making the two parallel edges very long and the rest of edges very short one can keep $\ell$ and $k$ arbitrarily large, yet make area go to zero. 

\end{remark}

\begin{remark}
Via Delzant's correspondence, Lemma~\ref{lemma:2dim} and Corollary~\ref{coro:2dim} imply a classification of (compact) toric 
integrable systems in dimension $4$. We refer to~\cite{PeVN09,PeVN11} for an analogue classification in the  \emph{semitoric} case, and~\cite[Theorem 1.3]{KPP1} and~\cite[Theorem 1.1]{KPP2} for a discussion of these concepts (fan, etc.) in that case.
The topology on the moduli space of semitoric integrable systems was defined by Palmer~\cite{Pa17}.
\end{remark}

\subsection{Infinitely many minimal Delzant $3$-polytopes}

In dimension three and higher, corner chopping alone can certainly not produce all Delzant polytopes from a finite list of combinatorial types of them. Consider, for example, the prism obtained as the Cartesian product $P\times I$ of a Delzant $m$-gon $P$ and a segment $I$. This is a Delzant $3$-polytope with $m+2$ facets, none of them triangular (unless $m=3$), hence it is not a corner chopping of any Delzant polytope with less facets.

However, in this example we can still say that if $m\geq 5$ then $P$ is a corner chopping of a Delzant $(m-1)$-gon $P'$, and then $P\times I$ is an ``edge chopping'' of $P'\times I$. Equivalently, we have that the normal fan of $P\times I$ is a blow-up of the fan of $P'\times I$ at a non-fulldimensional cone, in the sense of Definition~\ref{defi:blowup}.

The following two constructions show more drastically that one should not hope any analogue of Lemma~\ref{lemma:2dim} or Corollary~\ref{coro:2dim} 
in dimension three or higher.  In the first example we show not only that there are arbitrarily large $3$-dimensional unimodular fans that do not admit blow-downs, but also that such examples cannot be considered ``sporadic'', since any given fan can be refined to one of them. In the following proof we call \emph{valency} of a ray in a $3$-dimensional fan the number of full-dimensional cones containing it.

\begin{theorem}
\label{thm:dim3}
Let $\Sigma$ be any rational complete fan in $\R^3$. Then, there is a unimodular fan $\Sigma'$ that refines $\Sigma$ and such that $\Sigma'$ is not the blow up of any unimodular fan.
If $\Sigma$ is polytopal then $\Sigma'$ can be obtained polytopal too.
\end{theorem}

\begin{proof}
By Theorem~\ref{thm:desingular} we may assume without loss of generality that $\Sigma$ is already unimodular. We proceed in two steps, illustrated in  Figure~\ref{fig:sigmas}.

Let $\Sigma''$ be the  unimodular fan obtained from $\Sigma$ by blowing up all its $3$-dimensional cones. These blow-ups double the valency of each ray of $\Sigma$. Since all rays in a complete $3$-dimensional fan have valency at least three, all the original rays of $\Sigma$ have valency at least six in $\Sigma''$, while the rays introduced by the blow-ups have valency three.

We now define $\Sigma'$ subdividing each full-dimensional cone $\Cone\{\alpha_1,\alpha_2,\alpha_3\}\in \Sigma''$ into the following six cones:
\[
\begin{array}{ccc}
\Cone\big\{\alpha_1,\alpha_{12},\alpha_{13}\big\}, &
\Cone\big\{\alpha_2,\alpha_{12},\alpha_{23}\big\}, &
\Cone\big\{\alpha_3,\alpha_{13},\alpha_{23}\big\}, \\
\Cone\big\{\alpha_{123},\alpha_{12},\alpha_{13}\big\}, &
\Cone\big\{\alpha_{123},\alpha_{12},\alpha_{23}\big\}, &
\Cone\big\{\alpha_{123},\alpha_{13},\alpha_{23}\big\},
\end{array}
\]
where $\alpha_{ij}:=\alpha_i + \alpha_j$ and $\alpha_{123}:= \alpha_1 + \alpha_{2} + \alpha_{3}$. Since the $\alpha_{ij}$ and $\alpha_{123}$ are primitive and the
 six $3$-dimensional cones are unimodular, $\Sigma'$ is again a unimodular fan. 
The rays of $\Sigma''$ keep their valency in $\Sigma'$, while the new rays have valency three if they are of the form $\alpha_{123}$ and valency eight if they are of the form $\alpha_{ij}$. (For the latter, let $\alpha_{ij}$ be one of them; Since $\Sigma'$ is complete, $\alpha_{ij}$ lies in the common boundary of two three-dimensional cones of $\Sigma'$, and is contained in four of the six three-dimensional cones refining each). 

\begin{figure}[htb]
\includegraphics[height=4cm]{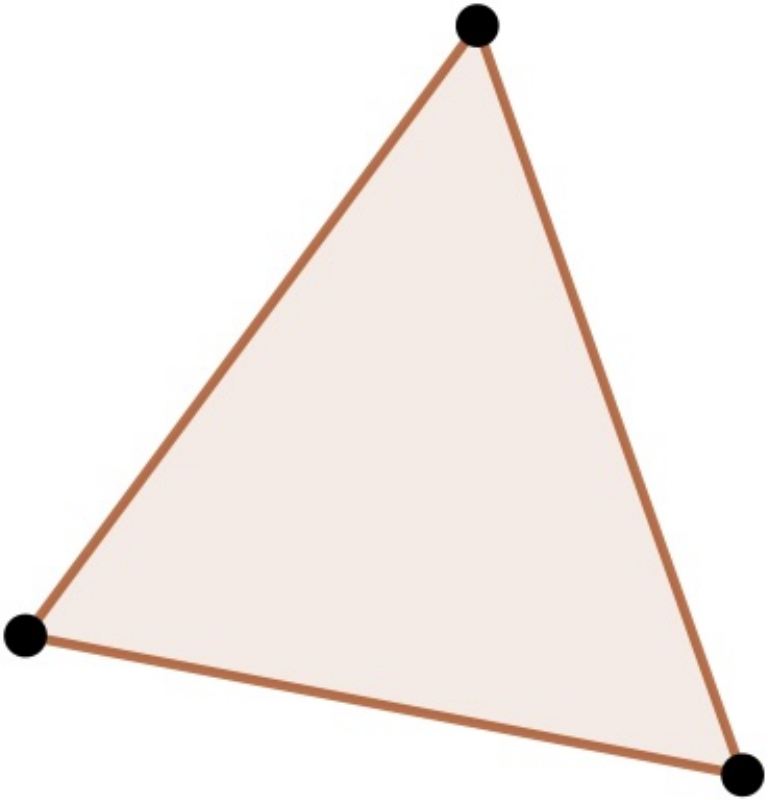}
\qquad
\includegraphics[height=4cm]{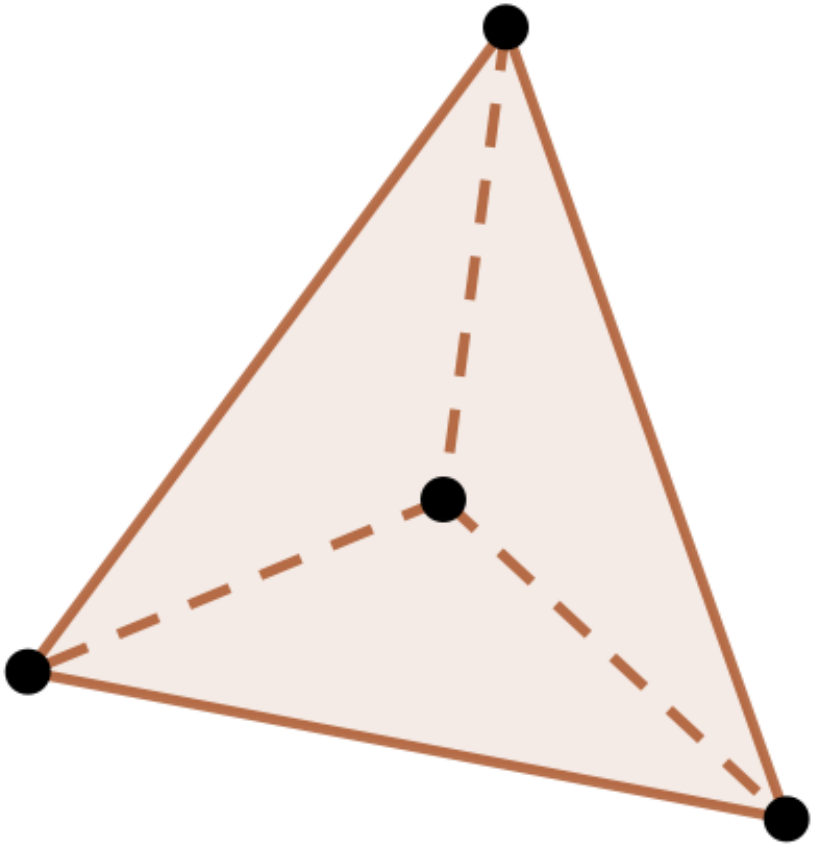}
\qquad
\includegraphics[height=4cm]{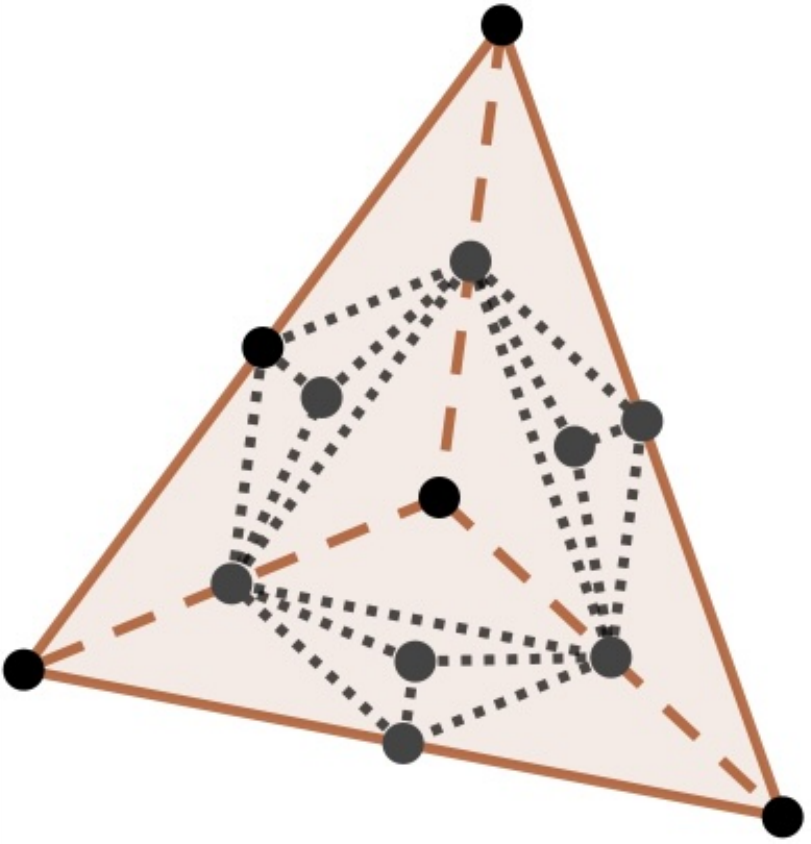}\\
$\Sigma$ \hskip 4.3cm
$\Sigma''$ \hskip 4.3cm 
$\Sigma'$  

\caption{The refinement constructed in the proof of Theorem~\ref{thm:dim3}.
Left: A $3$-dimensional cone of the original fan $\Sigma$, represented as a triangle, via an intersection with an affine plane.
Center: its blow-up in $\Sigma''$, consisting of three $3$-dimensional cones separated via three new two-dimensional cones (dashed lines) with one new ray (center dot).
Right: In $\Sigma'$ each $3$-cone from $\Sigma''$ is subdivided into six $3$-cones, introducing several new $2$-cones (pointed lines) and new rays of two types: one along each $2$-cone of $\Sigma''$, denoted $\alpha_{ij}$ in the proof, and one in each $3$-cone of $\Sigma''$, denoted $\alpha_{ijk}$ in the proof.
}
\label{fig:sigmas}
\end{figure}

To show that $\Sigma'$ is not the blow-up of any unimodular fan we observe that the new ray in a blow-up of a three-dimensional complete fan has valency three or four, depending on whether the blown-up cone is $3$-dimensional or $2$-dimensional. That is, only  rays of valency three or four  in $\Sigma'$ are candidates to be blown-down.
By the discussion above, $\Sigma'$ has no rays of valency four, and two types of rays of valency three:

\begin{enumerate}

\item The rays introduced in the blow-ups from $\Sigma$ to $\Sigma''$. Let $\alpha$ be one of them, and let $\beta,\gamma,\delta$ be its neighbors in $\Sigma''$, so that $\alpha = \beta+\gamma+\delta$. The three neighbors of $\alpha$  in 
$\Sigma'$ are 
$\alpha+\beta $,
$\alpha+\gamma$, and
$\alpha+\delta $.
Since their sum is $4\alpha$, $\alpha$ cannot be blown down. (Combinatorially it may look like it can, but the cone $\Cone\{\alpha+\beta , \alpha+\gamma , \alpha+\delta\}$ obtained by this supposed blow-down has multiplicity four, not one).

\item The rays of the form $\alpha_{ijk}$ introduced when going from $\Sigma''$ to $\Sigma'$. These have the same problem; blowing down $\alpha_{123}$ would produce the cone $\Cone\{\alpha_{12},\alpha_{12},\alpha_{23}\}$, of  multiplicity $2$.
\qedhere
\end{enumerate}
\end{proof}

In dimension two, a unimodular fan is a blow-up if and only if it is the refinement of another unimodular fan. 
In dimension three, the fans of the previous theorem do not admit blow-downs but they still are refinements of unimodular fans. 
Thus, one could perhaps hope that an analogue of Lemma \ref{lemma:2dim} can exist in dimension three except the refinements needed in the classification are a more complicated list than only blow-ups.
The following examples kill this hope, in any dimension $n\geq 3$:

\begin{theorem}
 \label{thm:isolated}
In any dimension $n\geq 3$ there are complete unimodular polytopal fans with arbitrarily many rays which are not proper refinements of any other unimodular fan.
\end{theorem}

\begin{proof}
We first construct them for $n=3$.
Let $\VV_k$ be the following set of $k+6$ rays in $\R^3$, for an arbitrary $k\in \N$:
\begin{align*}
&\alpha_+=(0,0,1),\quad
\alpha_-=(k+1,1,-1),\\
&\beta_{-2}=(0,-1,0),\quad
\beta_{-3}=(1,0,0),\\
&\beta_i=(i,1,0), \quad i\in \{-1,\dots,k\}.
\end{align*}
We consider the following complete simplicial fan $\NN_k$ in it, with $2k+8$ maximal cones:
\begin{align*}
&\Cone\big\{\beta_{-3}, \alpha_+,\alpha_-\big\},\quad
\Cone\big\{\beta_{k}, \alpha_+,\alpha_-\big\},\\
&\Cone\big\{\beta_{i-1}, \beta_i, \alpha_+\big\},\quad
\Cone\big\{\beta_{i-1}, \beta_i, \alpha_-\big\}, \quad i=-2,\dots,k.
\end{align*}
Observe that the configuration and the fan are invariant under the unimodular involution 
\begin{align*}
\R^3&\quad\longrightarrow\quad\R^3\\
(x_1,\ x_2,\ x_3)&\quad\longmapsto \quad(x_1+ (k+1)x_3, \ x_2+x_3, \ -x_3) 
\end{align*}
 which exchanges $\alpha_+$ and $\alpha_-$ and fixes the rest of $\VV_k$.

\begin{figure}[htb]
\includegraphics[height=4cm]{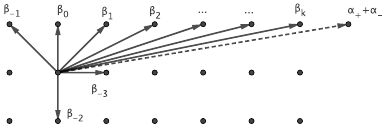}
\caption{The configuration $\VV_k$ and complete fan $\NN_k$ proving Theorem~\ref{thm:isolated}, pictured for $k=5$.
The picture shows it sliced with the hyperplane $H= \{x_3=0\}$ containing all generators except $\alpha_+$ and $\alpha_-$; the vector $\alpha_+ + \alpha_-\in H$ is not in $\VV_k$, but it is shown to represent the intersection of $H$ with the $2$-face $\Cone\{\alpha_+,\alpha_-\}\in \NN_k$. 
}
\label{fig:ops}
\end{figure}
One crucial property of this construction is that all the $\beta_i$  lie in the plane $H:=\{x_3=0\}$ and the other two generators $\alpha_+$ and $\alpha_-$ lie in opposite sides of it. 
This allows us to picture $\VV_k$ and $\NN_k$ in two dimensions, by intersecting them with $H$, as we do in Figure~\ref{fig:ops}.
In fact, the fan  $\NN_k$ can be understood as follows: start with the unimodular two-dimensional complete fan of Figure~\ref{fig:ops}  (including the ray $\alpha_++\alpha_-$), which is clearly unimodular. Take the cones
 of it with  both $\alpha_+$ and $\alpha_-$, which gives a unimodular complete fan in $\R^3$ with four maximal cones surrounding the ray $\alpha_+ +\alpha_-$. Then, blow-down  $\alpha_++\alpha_-$; that is, remove the four maximal cones containing it and insert $\Cone\{\beta_{-3}, \alpha_+,\alpha_-\}$ and $\Cone\{\beta_{k}, \alpha_+,\alpha_-\}$ instead.
This description shows that $\NN_k$ is indeed a complete unimodular fan.

We claim that if $\NN'$ is a complete unimodular fan refined by $\NN_k$ then $\NN'=\NN_k$. For this:
\begin{enumerate}
\item The rays $\alpha_+$ and $\alpha_-$ must be used in $\NN'$ since each of them is the only element of $\VV_k$ in a certain open half-plane (namely, the half-planes $x_3>0$ and $x_3<0$).
\item This implies that the cone $C=\Cone\{\alpha_+,\alpha_-\}$  must be in $\NN'$ too. Indeed,  $C$ must be contained in some cone $C'$ of $\NN'$ and since the generators of $C$ are used in $\NN'$ and $C'$ is simplicial, $C$ must be a face of $C'$.
\item The only $\beta_i$'s that form a unimodular cone together with $\Cone\{\alpha_+,\alpha_-\}$ are $\beta_{-3}$ and $\beta_{k}$. In order to check this, observe that the determinant of $\Cone\{\beta_i, \alpha_+,\alpha_-\}$ as a $3$-dimensional cone is the same as that of $\Cone\{\beta_i, \alpha_+ +\alpha_-\}$ as a $2$-dimensional one. We conclude that 
\[
\Cone\big\{\beta_{-3}, \alpha_+,\alpha_-\big\}, \ 
\Cone\big\{\beta_{k}, \alpha_+,\alpha_-\big\} \in \NN'.
\]

\item  Once we know this, all the other maximal cones must be of the form $\Cone\{\beta_i, \beta_j,\alpha_\pm\}$ for some $i,j$ and some choice of sign for $\alpha_\pm$. Hence, the rest of the fan must consist of two-dimensional cones in $H$ joined to both $\alpha_+$ and $\alpha_-$, and the cones used in $H$ must be unimodular themselves. An easy inspection of Figure~\ref{fig:ops} shows that the only way to complete the unimdular fan is exactly as $\NN_k$ does.
\end{enumerate}

Once the statement is proved for $n=3$ we prove it for any $n\geq 4$ by induction. Hence, we assume that for every $k$ there is a full-dimensional complete polytopal unimodular fan $\Sigma_k$ with at least $k$ rays in $\R^{n-1}$ that is not the proper refinement of any unimodular fan. 
We consider the \emph{suspension} of $\Sigma_k$ defined as:
\[
\operatorname{susp}(\Sigma_k):=
\Big\{\sigma\times \{0\} \,\Big|\, \sigma\in \Sigma_k\Big\}\cup
\Big\{\sigma\times [0,\infty) \,\Big|\,  \sigma\in \Sigma_k\Big\}\cup
\Big\{\sigma\times (-\infty,0] \,\Big|\,  \sigma\in \Sigma_k\Big\}.
\]
Put differently, from each cone $\Cone\{\alpha_1,\dots,\alpha_\ell\}$ in $\Sigma_k$ (including the $0$ cone, generated by the empty set) we have the following three cones in $\operatorname{susp}(\Sigma_k)$:
\[
\Cone\big\{\alpha'_1,\dots,\alpha'_\ell\big\}, \quad
\Cone\big\{\alpha'_1,\dots,\alpha'_\ell, {\rm e}_n\big\},\quad
\Cone\big\{\alpha'_1,\dots,\alpha'_\ell, -{\rm e}_n\big\},
\]
where $\alpha'_i:= \alpha_i \times \{0\}$, $i=1,\dots,\ell$.

$\operatorname{susp}(\Sigma_k)$ is an  $n$ dimensional complete  unimodular fan with $2$ more rays than $\Sigma_k$. It is not a refinement of any fan since any such fan would need to be of the form $\operatorname{susp}(\Sigma')$ where $\Sigma_k$ refines $\Sigma'$. It is polytopal since it is the normal fan of $P\times [0,1]$ for any polytope $P$ with normal fan $\Sigma_k$.
 \end{proof}

\begin{remark}
The example in the first part of the proof of Theorem~\ref{thm:isolated} is based on the \emph{one-point suspension} construction, described for example in \cite[Section 4.2.5]{triangbook}. We say that a configuration $\VV$ of dimension $n$ is a one-point suspension if there are two vectors $\alpha_+,\alpha_- \in \VV$ such that $\VV':=\VV\setminus \{\alpha_+,\alpha_-\} \cup \{\alpha_+ + \alpha_-\}$ is $(n-1)$-dimensional. Every simplicial fan $\NN$ with generators contained in $\VV$ (called a ``triangulation of $\VV$'' in the language of \cite{triangbook}) is obtained as the example in the proof: starting with a triangulation $\NN'$ of $\VV'$ do the following to each maximal cone $C$ of it:
\begin{itemize}
\item If $\alpha_+ + \alpha_-$ is a generator of $C$, substitute $\alpha_+ + \alpha_-$ by $\alpha_+$ and $\alpha_-$ as generators.
\item If $\alpha_+ + \alpha_-$ is not a generator of $C$, cone $C$ separately to  
$\alpha_+$ and to $\alpha_-$.
\end{itemize}
The suspension used in the second part of the proof is the case $\alpha_+ + \alpha_-=0$.
\end{remark}

\section{Topology of the space of Delzant polytopes} 
\label{sec:dp2}

In this section we study natural geometric structures (metric, topology) on the spaces of Delzant polytopes of any dimension.

\subsection{Two metrics and corresponding topologies on the space of convex bodies}
\label{sec:metric}

Following Gruber~\cite{Gruber} we denote by $\CC(n)$  the space of all convex bodies (i.e., compact convex sets) in $\R^n$ and by $\CC_p(n)$ the subset of \emph{proper} ones; that is,  those with non-empty interior or, equivalently, with positive volume. Also, we denote by $\PP(n)$  the subspace of perhaps not full-dimensional polytopes, and by $\PP_\Q(n)$  the subspace of polytopes with rational edge directions (or, equivalently by Lemma~\ref{lemma:rational}, those with rational facet normals).

\begin{definition}
We consider the following two  functions 
\[
\dist^V, \dist^H: \CC(n)\times\CC(n)\longrightarrow \R_{\geq 0},
\] 
called respectively the \emph{symmetric difference distance} and the \emph{Hausdorff distance}:
\begin{align*}
\dist^V (P,Q) & := \Vol(P\setminus Q) + \Vol(Q\setminus P), \\
\dist^H (P,Q) & := \max \Big\{\max_{x\in P} \distance(x,Q), \,\max_{y\in Q} \distance(y,P)\Big\}, 
\end{align*}
where $$\distance(x,Q)=\min_{y\in Q} \|x-y\|.$$
\end{definition}

Observe that for $\dist^V$ to be a distance function we need to restrict ourselves to $\CC_p$ since in it every two improper 
convex bodies are at distance zero. Alternatively, we can consider $\dist^V$  defined on $\CC_p\cup\{0\}$, where ``$0$'' denotes the class of all non-proper convex bodies.

The following is immediate. 

\begin{proposition}
The pairs $(\CC_p\cup\{0\},\dist^V)$  and $(\CC,\dist^H)$ are metric spaces.
\end{proposition}

Restricted to $\CC_p$,  $\dist^H$ and $\dist^V$ are not equivalent as distances; for example, the sequence of rectangles $[0,m]\times [0,1/m^2]$, with $m\in \N$ is a Cauchy sequence with respect to $\dist^V$ since the volume of the rectangles, hence of their symmetric differences, goes to zero. However, it is not Cauchy with respect to $\dist^H$ since the Hausdorff distance between any two of these rectangles is at least 1.

Pelayo\--Pires\--Ratiu\--Sabatini~\cite[Section~2.1]{PPRS14} consider the space of Delzant polytopes with the topology induced by $\dist^V$. We here will primarily work with $\dist^H$ since it relates better to normal fans. 
Fortunately, the two distances define the same topology, which gives us a well-defined space of Delzant polytopes:

\begin{proposition}[$\dist^H$ versus $\dist^V$] \label{prop:topology}
The following statements hold.
\begin{itemize}
\item[(1)]
Restricted to $\CC_p(n)$ the Hausdorff distance and the symmetric difference induce the same topology, which is locally compact.
\item[(2)]
$\CC_p(n)$ is not complete for any of the two distances. Its completion for $\dist^H$ is $\CC(n)$ and its completion for $\dist^V$ is 
$\CC_p(n)\cup \{0\}$. 
\end{itemize}
\end{proposition}

\begin{proof}
That $\dist^H$ and $\dist^V$  induce the same topology is proven in~\cite{SheWeb} (see also~\cite[Remark 10]{PPRS14}), and that the topology induced by $\dist^H$ is locally compact is \cite[Theorem 6.4]{Gruber}.

That the completion of $\dist^H$ is $\CC(n)$ appears in~\cite[Theorem 6.4]{Gruber}, where it is moreover proved that $\CC(n)$ is boundedly compact (stronger than being complete).

That the completion for $\dist^V$ is $\CC_p(n)\cup \{0\}$ can  be derived as follows (the case of $n=2$ is \cite[Theorem 7]{PPRS14}, with a different proof): let $(P_m)_{m\in \N}$ be a sequence of polytopes that is Cauchy in $\dist^V$. Since we only need to find a convergent subsequence, there is no loss of generality in assuming that for  every $m\in\N$ and $m'>m$
\[
\dist^V(P_m,P_{m'}) \leq 2^{-m}.
\]

For each $m$ let
\[
Q_m = \displaystyle\bigcup_{m' \geq m} P_{m'}.
\]
This is not a convex body but it is a measurable set and this is enough in order for $\dist^V$ to be well-defined. We have that 
\[
\dist^V(Q_m,P_m) \leq \sum_{m'>m} \dist^V(P_{m'},P_m) \leq 2^{-m+1}.
\]
Hence, the sequence $(P_m)_{m\in \N}$ converges if, and only if, the decreasing sequence $(Q_m)_{m\in \N}$ converges. We claim that the latter converges to
\[
Q= \overline{\bigcap_{m\in \N} Q_m},
\]
and that $Q$ either has volume zero or is a convex body. For the first part:
\begin{eqnarray*}
\dist^V(Q_m,Q) &=& \Vol(Q_m\setminus Q)   \\
&\leq & \sum_{m'\geq m} \Vol(Q_{m'}\setminus Q_{m'+1})  \\
&\leq & \sum_{m'\geq m} \Vol(P_{m'}\setminus P_{m'+1})  \\
&\leq & 2^{-m+1}. 
\end{eqnarray*}
For the second part, suppose that $Q$ has positive volume, so its interior is not empty. We first claim that the interior of $Q$ is convex. 
For this, let $p$ and $q$ be two interior points, so that there are balls 
${\rm B}(p,\varepsilon)$ and ${\rm B}(q,\varepsilon)$ contained in $Q$, that is, contained in infinitely many of the original $P_m$. 
Since the original sequence was Cauchy, ${\rm B}(p,\varepsilon/2)$ and ${\rm B}(q,\varepsilon/2)$ must be contained in all but perhaps finitely 
many of the $P_m$ and, by convexity of $P_m$, the same happens for the segment $pq$.

We now claim that $Q$ equals the closure of its interior, and hence it is convex. For this, let $q\in Q$ and $p\in Q^\circ$. Choosing as before a ball 
${\rm B}(p,\varepsilon)$ contained in $Q$, we have that $q$ lies in the closure of the interior of $\Conv\left\{{\rm B}(p,\varepsilon)\cup\{q\}\right\}$, which in turn is contained in the closure of the interior of $Q$.

Finally, since $Q$  is a closed convex set with non-empty interior and of finite volume, it must be bounded, hence a convex body.
\end{proof}

In what follows we denote by $\widetilde{\CC(n)}, \widetilde{\CC_p(n)}$ respectively the quotients of $\CC(n), \CC_p(n)$ by the action of $\AGL(n,\Z)$, with the quotient topology, 
and consider the map  
\begin{align} \label{e:inf}
\widetilde{\dist^V} \colon \widetilde{\CC_p(n)} \times \widetilde{\CC_p(n)} &\longrightarrow\qquad \R_{\geq 0} \nonumber \\
([P], [Q])\quad&\longmapsto \inf_{\substack{P'\in [P],\\ Q'\in [Q]}} \dist^V(P',Q'),
\end{align}
where $[P]$ and $[Q]$ denote the orbits of $P$ and $Q$.

\begin{lemma}
\label{l:bounded}
For any proper convex body $P \in \CC_p(n)$ and any bounded set $K$ there are only a finite amount of elements $\sigma\in \AGL(n,\Z)$ such that $\sigma(P) \subset K$.
\end{lemma}

\begin{proof}
Since $P$ is proper, there is an $N\in \N$ such that $P$ contains an affine basis $B$ of $\frac1N\Z^n$. Now, $\frac1N\Z^n$ is invariant under $\AGL(n,\Z)$ and knowing $\sigma|_B$ completely determines each $\sigma\in \AGL(n,\Z)$, so the number of possible $\sigma$ with $\sigma(P)\subset K$ is bounded above by the number of $(n+1)$-tuples of points of $\frac1N\Z^n$ in $K$, which is finite.
\end{proof}

The following example, which as far as we know is new, shows that Lemma~\ref{l:bounded} does not hold for improper convex bodies.  In particular,
it shows that the quotient topology on $\widetilde{\CC(n)}$ is not Hausdorff.

\begin{proposition}
\label{prop:fibonacci}
Any neighborhood of the origin contains infinitely many distinct $\AGL(2,\Z)$ images of the segment
 $S$ joining $(0,0)$ and $(\phi,-1)$.
\end{proposition}

\begin{proof}
Consider the unimodular matrices
\[
\begin{pmatrix}
  F_k & F_{k+1}\\
  F_{k+1} & F_{k+2}
\end{pmatrix}
\]
where the $F_k$ are the Fibonacci numbers. They take $S$ to segments of lengths going to zero as $k$ goes to infinity, so 
that there is an infinite amount of them in every neighborhood of the origin.
\end{proof}

The following result shows that the infimum in expression (\ref{e:inf}) is really a minimum and it implies that the distance between different orbits is positive. This is also proved in \cite[Proposition 4.6]{FuOh} but we want to emphasize that our proof is considerably shorter.

\begin{lemma}
\label{l:orbits}
Let $P, Q \in \CC_p(n)$ be proper convex bodies. Then $ \widetilde{\dist^V}([P],[Q])$
equals $\dist^V(P,Q')$ for some $Q'\in [Q]$. 
\end{lemma}

\begin{proof}
We first prove the following claim: let $m,M\in \R_{\geq 0}$, with $M>m$, and consider the boxes $B_m=[-m,m]^n$ and $B_M=[-M,M]^n$. Then, every convex body $K$ not contained in $B_M$ has at least a fraction of 
\[
(M-m)^n/(M+m)^n
\]
of its volume outside $B_m$. 

To prove the claim, let $p=(p_1,\ldots, p_n) \in K\setminus B_M$, so that at least one coordinate of $p$ is strictly larger than $M$ in absolute value. Suppose without loss of generality that $p_1>M$. Let 
\[
K_0 = K \cap \Big\{(x_1,\ldots, x_n) \in \R^n \,\Big|\, x_1=m\Big\},
\]
and let $K_1$ be the following bounded affine cone with apex at $p$:
\[
K_1:= \Big\{p+ \lambda (q-p)  \,\Big|\ q\in K_0, \lambda \geq 0\Big\} \cap \Big\{(x_1,\ldots, x_n) \in \R^n \,\Big|\, -m\leq x_1 \Big\}. 
\]
Observe that the fraction of volume of $K_1$ contained in $\{(x_1,\ldots, x_n) \in \R^n \,|\, x_1\geq m\}$ is 
exactly $(M-m)^n/(M+m)^n$. The claim follows then from the fact that 
\begin{eqnarray}
& K\cap B_m \subset K_1 \cap \Big\{(x_1,\ldots, x_n) \in \R^n \,\Big|\, x_1 \leq m\Big\}, \nonumber \\
& K\cap (B_M\setminus B_m) \supset K_1 \cap \Big\{(x_1,\ldots, x_n) \in \R^n \,\Big|\, x_1 \geq m\Big\}. \nonumber
\end{eqnarray}

To prove the lemma first observe that 
\[
\dist^V(P',Q') = \Vol(P') + \Vol(Q') - 2\Vol(P'\cap Q').
\]
Since volume is invariant under  $\AGL(n,\Z)$, the first two terms are independent of the representatives $P'$ and $Q'$ chosen. Hence, we only need to understand 
\[
\sup_{P'\in [P], Q'\in [Q]} \Vol(P'\cap Q').
\]
Invariance under the action of $\AGL(n,\Z)$ also shows that this supremum equals
\[
\sup_{Q'\in [Q]} \Vol(P \cap Q').
\]
If $\Vol(P,Q')=0$ for every $Q'\in [Q]$ then there is nothing to prove, so we assume that we have a representative $Q_0\in [Q]$ 
with $\Vol(P \cap Q_0)>0$. Let $m,M\in \R$ be such that $P$ is contained in the box $o + B_m$ for some point $o$ and such that  
\[
\frac{(M-m)^n}{(M+m)^n} \geq \frac{\Vol( Q) - \Vol(P \cap Q_0)}{\Vol(Q)}.
\]
Such an $m$ exists since $P$ is bounded, and $M$ exists since the left-hand side in this equation goes to $1$ as $M$ goes to infinity with $m$ fixed, while the right-hand-side is smaller than 1. Observe that this right-hand side equals the fraction of volume of $Q_0$ outside $P$.

By the claim, any convex body not contained in $o+B_M$ has less fraction of its volume intersecting $o+B_m$ (and hence even 
less intersecting $P$) than $Q_0$. Now, every $Q'$ in $[Q]$ has the volume of $Q$ and only finitely many of them are contained 
in $B_M$ by Lemma~\ref{l:bounded}, so only finitely many polytopes in $[Q]$ can have a larger intersection with $P$ than $Q_0$.
\end{proof}
}

\begin{corollary}[\protect{\cite[Theorem 3.2]{FuOh}}]
\label{c:quotient}
$(\widetilde{\CC_p(n)} , \widetilde{\dist^V})$ is a metric space.
\end{corollary}

\begin{proof}
Lemma~\ref{l:orbits} implies that different orbits $[P]$ and $[Q]$ have strictly positive distance $ \widetilde{\dist^V}([P],[Q])$, and it also implies the triangular inequality:
\begin{align*}
\widetilde{\dist^V}([P],[Q]) + \widetilde{\dist^V}([Q],[R]) &=
\inf_{P'\in [P]} \dist^V(P',Q) +
\inf_{R'\in [R]} \dist^V(Q, R') \\ 
&= \inf_{\substack{P'\in [P],\\ R'\in [R]}} (\dist^V(P',Q) +\dist^V(Q, R') )\\
&\geq \inf_{\substack{P'\in [P],\\ R'\in [R]}} (\dist^V(P',R') ) \\
&=\widetilde{ \dist^V}([P],[R]).
\end{align*}
\end{proof}

\begin{remark}
There are several obstacles to generalize these results on the quotient metric to the Hausdorff distance $\dist^H$. The first one is that one would like (at least if we want to understand the completion of $\widetilde{\CC(n)}$ as we do next) to deal also with improper convex bodies but, as shown by Proposition~\ref{prop:fibonacci}, the quotient space $\widetilde {\CC(n)}$ is not Hausdorff.

But even for proper convex bodies, where the quotient topology of the metric spaces $(\CC(n), \dist^V)$ and $(\CC(n), \dist^H)$ coincides, the distance $\widetilde{\dist^H}$ between two different orbits can be zero. Consider for example the standard triangle $P=\Conv\{(0,0),(1,0),(0,1)\}$ and the unit square $Q=[0,1]^2$. Then 
\[
P'=\Conv\big\{(0,0),(1,0),(k,1)\big\}\ \text{ and }\ Q'=\Conv\big\{(0,0),(1,0), (k-1,1),(k,1)\big\}
\]
are representatives with $\dist^H(P',Q') =1/\sqrt{k^2+1}$, which goes to zero as $k$ goes to infinity.
\end{remark}

Related to this,  
Macbeath proved~\cite[Theorem~1]{Macbeath} that the quotient of $\CC_p(n)$ by the action of $\AGL(n,\mathbb{R})$ is compact and metrizable, with the distance between $P$ and $Q$ defined essentially as $|\log (\Vol(Q'')/\Vol(Q'))|$ where $Q'$ and $Q''$ are the biggest and smallest (respectively) elements of the orbit of $Q$ with $Q'\subset P\subset Q''$.

The following is an analogue of Proposition~\ref{prop:topology} for quotient spaces by the $\AGL(n,\Z)$\--action. Part (1) is \cite[Theorem 3.3]{FuOh}, and part (2) allows us to solve problem 5.2 in the same paper (Theorem~\ref{thm:quotientD}).

\begin{proposition}
\label{prop:quotientC}
The following statements hold.
\begin{enumerate}
\item[{\rm (1)}]
The metric topology of $(\widetilde{\CC_p(n)}, \widetilde{\dist^V})$   coincides with the quotient topology
 induced on $\widetilde{\CC_p(n)}$ by the action of $\AGL(n,\Z)$ on $\CC_p(n)$. 
\item[{\rm (2)}]
The metric space $(\widetilde{\CC_p(n)}, \widetilde{\dist^V})$ is not complete. Its completion is $\widetilde{\CC_p(n)}\cup \{0\}$.
\end{enumerate}
\end{proposition}

\begin{proof}
In order to show that the topologies coincide we need to show that the identity map between the quotient  and  metric topologies in $ \widetilde{\CC_p(n)}$ is continuous and open. 
For this, observe that
a ball ${\rm B}([P], \varepsilon)\subset \widetilde{\CC_p(n)}$ with respect to the metric $\widetilde{\dist^V}$ lifts, by Lemma~\ref{l:orbits}, to 
$\cup_{P'\in [P]}{\rm B}(P,\varepsilon)$; since the latter is open in $\CC_p(n)$ the former is open in the quotient topology. 
Conversely, if a $\widetilde U\subset \widetilde{\CC_p(n)}$ is open in the quotient topology then $\widetilde U$ lifts to an open and $\AGL(\Z,n)$-invariant set $U$ in $\CC_p(n)$. 
For each $P$ with $[P]\in \widetilde U$ we have that there is a ball ${\rm B}(P,\varepsilon)$ contained in $U$ and, by invariance of $U$, the same $\varepsilon$ satisfies 
that ${\rm B}(P',\varepsilon) \subset U$ for every $P'\in [P]$. 
Hence, $\widetilde U$ contains the ball ${\rm B}([P], \varepsilon)$.

For the completion, we extend $\widetilde{\dist^V}$ to $\widetilde{\CC_p(n)}\cup \{0\}$ defining
\[
\widetilde{\dist^V}([P],0) = \Vol(P).
\]
Since we know ${\CC_p(n)}\cup \{0\}$ to be the completion of $(\CC_p(n),\dist^V)$ (Proposition~\ref{prop:topology}, (2)) we only need to show that every Cauchy sequence in $\widetilde{\CC_p(n)}$ lifts to a Cauchy sequence in $\CC_p(n)$.

Let $([P_k])_{k\in \N}$ be a Cauchy sequence and, by taking a subsequence if needed, assume that
\[
\widetilde{\dist^V} ([P_k],[P_{k+1}]) \leq 2^{-k},\quad \forall k\in \N.
\]
Thanks to Lemma~\ref{l:orbits}, we can then choose one by one representatives $P'_k$ in each orbit $[P_k]$ in such a way that for every $k$ we have 
\[
{\dist^V} (P'_k,P'_{k+1})  = \widetilde{\dist^V} ([P_k],[P_{k+1}]) \leq 2^{-k},\quad \forall k\in \N,
\]
which gives us a Cauchy sequence $(P'_k)_{k\in \N}$ in $\CC_p(n)$.
\end{proof}

\subsection{Minkowski paths prove that $\Dt$ is connected}
\label{sec:connected}

We now formally introduce the spaces of Delzant polytopes.

\begin{definition} \label{d:spaces}\ 
\begin{itemize}
\item[{\rm (i)}]
 We denote by $\Dt$ the space of all Delzant $n$-polytopes in $\R^n$, with the topology induced by either $\dist^H$ or $\dist^V$. 
 \item[{\rm (ii)}]
We denote by $\Dtt$ the quotient of $\Dt$ by the action of $\AGL(n,\Z)$, with the quotient topology.
\end{itemize}
\end{definition}

In order to analyze the topology of $\Dt$ and its relation to $\CC(n)$ the \emph{Minkowski sum} operation is very useful. Recall that for given sets $P, Q\subset \R^n$ their Minkowski sum equals
\[
P+Q:=\Big\{x+y \,\Big|\, x\in P, y\in Q\Big\}.
\]
Minkowski sum preserves convexity, boundedness, and compactness so it is well-defined on $\CC(n)$. Together with Minkowski sum we consider multiplication by a scalar and, more generally, \emph{Minkowski combinations}. That is, for given $\lambda,\mu \in [0,\infty)$ and $P,Q\in\CC(n)$ we consider
\[
\lambda P+\mu Q:=\Big\{\lambda x+\mu y \,\Big|\, x\in P, y\in Q\Big\}.
\]

Minkowski sum is very well behaved with respect to \emph{support functions}. Recall that the support function ${\rm h}_C: (\R^n)^*\to \R$ of a convex body $C\subset \R^n$ is defined by%
\[
{\rm h}_C(\alpha):= \max_{x\in C} \langle \alpha,x\rangle.
\]

\begin{proposition}[\protect{\cite[Theorem 6.1, Proposition 6.2, Theorem 4.3]{Gruber}}]\ 
\label{prop:minkowski}
\begin{itemize}
\item[(1)] With the operations of Minkowski addition and multiplication by a scalar $\CC(n)$ is a semigroup (with identity $\{0\}$) and a cone.
\item[(2)] ${\rm h}_{\lambda P+\mu Q} = \lambda {\rm h}_P+\mu {\rm h}_Q$.
\item[(3)] A function $h:(\R^n)^*\to \R$ is the support function of some (and then a unique) convex body if and only if 
it is positive homogeneous and subadditive. That is, if and only if 
\[
\forall \alpha,\beta\in (\R^n)^* \text{ and }\ \forall\lambda,\mu \geq 0:\quad
h(\lambda\alpha) = \lambda h(\alpha), \quad h(\lambda \alpha+\mu\beta) \leq  \lambda h(\alpha) + \lambda h(\beta).
\]
\end{itemize}
\end{proposition}

That is to say, algebraically, Minkowski combinations allow us to think of $\CC(n)$ as the subcone of positive homogeneous and subadditive functions in the vector space of functions from $\R^n$ to $\R$. This analogy carries over to the topology of $\CC(n)$ under the following setting.  
Let us call \emph{uniform norm}, and denote $\|f\|_\infty$, of a positively homogeneous  continuous function $f:(\R^n)^*\to \R$ the supremum norm restricted to the unit sphere. That is:
\[
\|f\|_\infty := \sup_{\|\alpha\|_2 =1}{|f(\alpha)|}.
\]
Then we have that:

\begin{theorem}[\protect{\cite[Theorem 1.8.11]{Schneider93}}]\
\label{thm:hausdorff}
For any two convex bodies $C,D\in \CC(n)$ with support functions ${\rm h}_C, {\rm h}_D : (\R^n)^* \to \R$ we have
\[
\dist^H(C,D) = \|{\rm h}_C - {\rm h}_D\|_\infty = \max_{\|\alpha\|_2 =1}\left(\max_{x\in C}\langle \alpha,x\rangle - \max_{x\in D}\langle \alpha,x\rangle \right).
\]
\end{theorem}

\begin{corollary}
Let ${\rm C}(\S^{n-1})$ be the space of continuous functions from the unit sphere $\S^{n-1}\subset\R^n$ to $\R$, with the supremum norm. Then, 
$\CC(n)$ is homeomorphic (in fact, isometric) to the subspace of ${\rm C}(\S^{n-1})$ consisting of functions whose positive homogeneous extension is subadditive.
\end{corollary}

The support function of a polytope is linear within each normal cone and, in fact, this characterizes polytopes: a convex body $C$ is a polytope if and only if its support function ${\rm h}_C$ is piecewise linear.
Together with the fact that ${\rm h}_{P+Q} = {\rm h}_P + {\rm h}_Q$, this implies that when $P$ and $Q$ are polytopes $P+Q$ is also a polytope and the normal fan of $\NN(P+Q)$ is the common refinement of $\NN(P)$ and $\NN(Q)$ (that is, cones in $\NN(P+Q)$ are the intersections of a cone from $\NN(P)$ and one from $\NN(Q)$).

The main property of Minkowski sums that we need is the following:

\begin{lemma}
For every $P,Q\in\CC(n)$ the map
\begin{align*}
[0,\infty)^2 & \longrightarrow\CC(n)\\
(\lambda, \mu) &\longmapsto \lambda P+\mu Q,
\end{align*}
is Lipschitz with respect to $\dist^H$, hence uniformly continuous.
\end{lemma}

\begin{proof}
Let $K$ be the radius of a ball centered at the origin and containing both $P$ and $Q$ and consider two Minkowski sums $A=\lambda P + \mu Q$ and $B=\lambda'P + \mu'Q$, with $\lambda, \lambda',\mu,\mu'\in [0,\infty)$. 

Assume that the Hausdorff distance $\dist^H(A,B)$ is attained by the distance to $B$ from a certain point $\lambda x +\mu y \in A$, where $x\in P$ and $y\in Q$. Since the point $\lambda' x +\mu' y $ is in $B$ we have that
\begin{eqnarray*}
\dist^H(A, B) &= &\dist^H(\lambda P + \mu Q, \lambda' P+\mu' Q)  \\
&=&\distance(\lambda x +\mu y, \lambda' P+\mu' Q)  \\
&\leq &\|\lambda x +\mu y - \lambda' x -\mu' y\|  \\
&\leq &|\lambda-\lambda'| \, \|x\| +|\mu-\mu'| \, \|y\|  \\
&\leq &K (|\lambda-\lambda'| \,  +|\mu-\mu'| \,) \\
&=&K \|(\lambda,\mu)-(\lambda',\mu')\|_1,
\end{eqnarray*}
where the last equality is by definition of the norm $\|\cdot\|_1$, which is equivalent to the standard norm in $[0,\infty)^2]$.
\end{proof}

Throughout this paper a \emph{path} in a topological space $X$ is a continuous map $f:[a,b] \to X$
from a compact interval $[a,b]\subset \R$ , with $a<b$. The points $f(a), f(b)\in X$ are the extreme points of the path and we often say that $f$ is  \emph{a path from $f(a)$ to $f(b)$}.

\begin{corollary}
\label{coro:path}
For every $P,Q\in\CC(n)$ the map
\begin{align*}
[0,1] & \longrightarrow\CC(n)\\
\lambda &\longmapsto \lambda P+(1-\lambda) Q
\end{align*}
is a path from $P$ to $Q$ in $\CC(n)$.
\end{corollary}

\begin{definition}
We call the path of Corollary~\ref{coro:path} the \emph{Minkowski segment} from $P$ to $Q$ and denote it by $PQ$. 
\end{definition}

The Corollary implies that $\CC(n)$ is path-connected and, in fact, contractible. (For the latter, fix a convex body $K$ and consider the homotopy $(P,t) \mapsto (1-t)P + t K$). The same holds for any subset of $\CC(n)$ closed under Minkowski addition and scaling, for example for $\CC_p(n)$,  $\PP(n)$, and  $\PP_\Q(n)$.

It does not directly imply that $\Dt$ is path-connected, since the Minkowski sum of two Delzant polytopes is in general not Delzant, or even simple. E.g., the Minkowski sum of the standard triangle with vertices $(0,0)$, $(1,0)$, $(0,1)$ and its reflection on the vertical axis  is a non-Delzant pentagon.   However, there is an easy way around this issue.

\begin{lemma}
\label{lemma:connected}
Let $P$ and $Q$ be two polytopes in $\R^n$ and suppose that the normal fan of $P$ refines that of $Q$ then 
\[
\NN(\lambda P + (1-\lambda) Q)) = 
\begin{cases}
\NN(P) \text{ if $\lambda\ne 0$},\\
\NN(Q) \text{ if $\lambda = 0$}.\\
\end{cases}
\]
In particular, if $P$ and $Q$ are both in $\Dt$ then the Minkowski segment between them stays in $\Dt$.
\end{lemma}

\begin{proof}
$\NN(\lambda P+(1-\lambda)Q)$ is the common refinement of $\NN(\lambda P)$ and $\NN((1-\lambda)Q)$.
\end{proof}

Part (2) of our next Theorem~\ref{thm:connected} and part (1) of Theorem~\ref{thm:quotientD} were proved in \protect{\cite[Theorem 2]{PPRS14}} for the case of $n=2$  using the classification Corollary~\ref{coro:2dim}. We here extend both results to arbitrary $n$ and to $\Dtt$.

\begin{theorem} 
\label{thm:connected}\ 
\begin{itemize}
\item[(1)] $ \Dt$ is path-connected.
\item[(2)] $\Dt$ is not locally compact, but it is dense in $\CC_p(n)$. Hence, its completion is the same as that of $\CC_p(n)$ (namely $\CC_p(n)\cup\{0\}$ for $\dist^V$ and $\CC(n)$ for $\dist^H$).
\end{itemize}
\end{theorem}

\begin{proof}
In order to prove part (1) we simply note that for given $P_1,P_2\in \Dt$ there is, by Corollary~\ref{coro:desingular}, a Delzant polytope $Q$ whose normal fan refines that of $P_1+P_2$, hence it refines the normal fans of both $P_1$ and $P_2$. By Lemma~\ref{lemma:connected}, concatenating the Minkowski segments from $P_1$ to $Q$ and from $Q$ to $P_2$ we get a path from $P_1$ to $P_2$ that stays in $\Dt$.

In part (2), the fact that $\Dt$ is neither complete nor locally compact is proved in \protect{\cite[Theorem 2]{PPRS14}} for $n=2$ and the proof carries over for higher $n$. 

In order to show that $\Dt$ is dense in $\CC_p(n)$ let $K\in \CC_p(n)$ and let us show how to approach $K$ via Delzant polytopes, in two steps:

\begin{itemize}
\item \emph{$K$ can be approached via (perhaps not Delzant) rational polytopes}. This is classical and easy to prove. For example, let 
\[
K_m:= \Conv\left(K\cap \frac1m\Z^n\right),
\]
which is a rational polytope contained in $K$ for each $m\in \N$. Since every neighborhood of each $x\in K$ contains points of $K_m$ for every sufficiently large $m$, we have that 
\[
\lim_{m\to \infty}\dist^H(K_m,K) =0.
\]
Here we are implicitly using compactness of $K$, which allows us to cover $K$ by finitely many arbitrarily small neighborhoods of points $x\in K$.

\item \emph{Every rational polytope $Q$ can be approached by Delzant polytopes}. Let $P$ be a Delzant polytope such that $\NN(P)$ refines $\NN(Q)$, which exists by Theorem~\ref{thm:desingular}. Then the path $\lambda P + (1-\lambda)Q$ from $P$ to $Q$ consists only of Delzant polytopes with the normal fan of $P$, except at its origin $\lambda=0$ where it equals $Q$.
\qedhere
\end{itemize}
\end{proof}

\begin{theorem} 
\label{thm:quotientD}\ 
\begin{itemize}
\item[(1)] The metric topology of $(\Dtt, \widetilde{\dist^V})$ coincides with the quotient topology of $\Dt$ by the action of $\AGL(n,\Z)$.
\item[(2)] $ \Dtt$ is path-connected.
\item[(3)] $(\Dtt, \widetilde{\dist^V})$ is not locally compact, but it is dense in $\widetilde{\CC_p(n)}$. 
Hence, its completion is the same as that of $\widetilde{\CC_p(n)}$, namely $\widetilde{\CC_p(n)}\cup\{0\}$.
\end{itemize}
\end{theorem}

\begin{proof}
Part (1) follows from the same result for $\CC(n)$ and $\widetilde{\CC(n)}$ (Proposition~\ref{prop:quotientC}) and part (2) is obvious from path-connectedness of $\Dt$.

Part (3) is also immediate from the same result for $\Dt$ together with Lemma~\ref{l:orbits}.
\end{proof}

The last part of  Theorem~\ref{thm:quotientD} solves a problem posed by  Fujita\--Ohashi~\cite[Problem~5.2]{FuOh} in 2018. 
We later solve Problems 5.3 and~5.5 of the same paper.

\subsubsection{Local and global complexity: comments on the proof of Theorem~\ref{thm:connected}} \label{sec:pm}

Let us call a path in $\CC(n)$ a \emph{Minkowski path} if it is a finite concatenation of Minkowski segments.

Our proof of connectedness for $\Dt$ shows that any two Delzant polytopes $P_1$ and $P_2$ can be connected by a Minkowski path $P_1QP_2$ consisting of two segments. Metrically, this path can be made as close to the Minkowski segment from $P_1$ to $P_2$ as one wishes, via the following trick. As in the proof of the theorem, let $Q$ be a Delzant polytope whose normal fan refines those of $P_1$ and $P_2$. Now, for each $\varepsilon\in (0,1)$ let
\[
Q_\varepsilon= (1-\varepsilon) P_1 + \varepsilon Q,
\]
which has the same normal fan as $Q$ but is very close to $P_1$ if $\varepsilon$ is very small. Then, the  path $P_1 Q_\varepsilon  P_2$, which is contained in $\Dt$, is very close to the Minkowski segment  $P_1 P_2$.

However, our proof does not give a bound for the following two measures of complexity.

\subsubsection*{Local complexity}

Let us call \emph{complexity of the Minkowski segment  $PQ$} the number of rays of $\NN(P)$ that are not in $\NN(Q)$ plus the number of  rays of $\NN(Q)$ that are not in $\NN(P)$. Equivalently, this is the number of rays of the intermediate normal fan (which is the same in the interior of the path) that are not common to $P$ and $Q$. Let us call \emph{local complexity of a Minkowski path} the maximum complexity of its Minkowski segments.

For example, a Minkowski segment between two Delzant polytopes has complexity $1$ if and only if $P$ is a blow-up of $Q$, or viceversa. 
Hence, a Minkowski path has local complexity 1 if and only if it is a sequence of blow-ups and blow-downs.
Our proof of Theorem~\ref{thm:connected} does not give any bound on the local complexity of the path, but the following result does:

\begin{theorem}[Weak Factorization Theorem \protect{\cite[Theorem A]{Wlodarczyk}}, see also \cite{AKMW}]
\label{thm:factorization}
Every two unimodular fans can be joined by a sequence of blow-downs and blow-ups.
\end{theorem}

\begin{corollary}
\label{coro:factorization}
Every two Delzant polytopes can be joined by a Minkowski path of local complexity one.
\end{corollary}

Observe that our proof of path-connectedness is elementary and self-contained, while the Weak Factorization Theorem and its proof are far from elementary, as shown by its history of claimed proofs that turned out to be incorrect.

\begin{remark}
There is also the Strong Factorization statement, asserting that every two unimodular fans can be joined by a sequence of blow-ups followed by a sequence of blow-downs. This was conjectured by Hironaka (1960)~\cite{Hironaka} and Oda (1978)~\cite{OdaMiyake}. Morelli (1996)~\cite{Morelli} claimed to have a proof, but it was later discovered that the proof had gaps.
The recent preprint~\cite{AdiPak} contains a new proof.
\end{remark}

In dimension 2, Corollary~\ref{coro:factorization} easily follows from Lemma~\ref{lemma:2dim}, as follows:
\begin{enumerate}
\item By the Lemma, we can go from any Delzant polygon other than a triangle to a quadrilateral by a sequence of blow-downs. Such quadrilaterals are $\GL(\Z,2)$-equivalent to Hirzebruch quadrilaterals ${\rm H}_{a,b,k}$.

\item The Hirzebruch quadrilaterals with consecutive parameters $k$ and $k+1$ can be joined by a blow-up followed by a blow-down (see~\cite[Figure 4]{PPRS14}).

\item The Hirzebruch quadrilateral with $k=1$ is a blow-up of the standard Delzant triangle 
\[
\Big\{(x_1,x_2) \,\Big|\, x_1\geq 0, x_2\geq 0, x_1+x_2\leq 1\Big\}.
\]

\item Any two different Delzant triangles can be joined by a sequence of blow-ups and blow-downs. This follows easily from the fact that $\GL(2,\Z)$ is generated by the transposition of coordinates plus the  matrix $\begin{pmatrix}1&0\\1&1\end{pmatrix}$. 
\end{enumerate}

In fact, these statements (without the last one) are essentially the proof of path-connectedness of $\Dtt$ in \cite[Theorem 5]{PPRS14}.

\subsubsection*{Global complexity}
By \emph{global complexity of a path} we mean the total number of rays used along the whole path apart of those at the end-points. The global complexity of the Minkowski path from any Delzant polygon to the quadrilateral that it refines is zero, so only the path between two quadrilaterals needs to be considered.

Let $P$ and $Q$ be the two Delzant quadrilaterals that we want to join by a path, and consider their normal fan $P+Q$. In dimension two there is a unique minimal way to desingularize $\NN(P+Q)$, namely introducing the rays that belong to the \emph{Hilbert basis} of each cone. The size of the Hilbert basis  of a two-dimensional rational cone is bounded by its determinant, so we can get a path from $P$ to $Q$ with global complexity bounded by the sum of the determinants of the cones of $P+Q$. This can be arbitrarily large if $P$ and $Q$ are very ``skew'' with respect to one another, but we can consider this skewness to be the natural parameter if we want to measure the ``global distance'' between $\NN(P)$ and $\NN(Q)$.

\subsection{The stratification of $\Dt$ by normal equivalence}
\label{sec:stratification}

\subsubsection{The secondary fan of a vector configuration, and the secondary cone of a fan $\NN$.}
We say that two polytopes $P$ and $Q$ are \emph{normally equivalent} if they have the same normal fan. Normal equivalence gives a natural stratification of the space of all polytopes. Restricted to $\PP_\Q(n)$ and hence also to $\Dt$, it has countably many strata. 

To make this more clear, let $\VV= \{\alpha_1,\dots,\alpha_m\} \subset (\R^n)^*$ be a finite number of vectors \emph{positively spanning} $(\R^n)^*$, that is, with $\Cone \VV=(\R^n)^*$.
We will later assume that the $\alpha_i$ are integer and primitive, but for the time being we don't need that.

Let us denote by $\PP(\VV)\subset\PP(n)$ the space of all polytopes $\{{\rm P}(\VV,b)\,|\,b\in \R^\VV\} $ where
\begin{equation}
{\rm P}(\VV,b) := \Big\{x\in \R^n\,\Big|\, \langle \alpha_i, x\rangle \leq b_i, \ i\in \{1,\dots,m\}\Big\},
\label{eq:polytope}
\end{equation}
and $\R^\VV\cong \R^m$ denotes the space of possible right-hand sides of the defining inequalities.
$\PP(\VV)$ contains the full-dimensional polytopes whose facet normals are contained in $\VV$, but it also contains lower-dimensional polytopes and, in fact, it contains all points. Indeed, for each $p\in \R^n$, taking $b_i=\langle \alpha_i, p\rangle$ we have that ${\rm P}(\VV,b)=\{p\}$.

Not every vector $b\in \R^\VV$ defines a polytope, since the corresponding system of inequalities may be infeasible. 
Also, if one of the equations in the system defining ${\rm P}(\VV,b)$ happens to be redundant then increasing the value of the corresponding $b_i$ does not change the polytope. Hence,  we can assume without loss of generality that the $b$ we use equals the support function of ${\rm P}(\VV,b)$ at the vectors $\alpha_i$ of $\VV$. 
With this assumption we have a bijection between $\PP(\VV)$ and the subset of $\R^\VV$ obtained by restricting to $\VV$ the support functions of polytopes in $\PP(\VV)$. Since we show below (Theorem~\ref{thm:cells}) that this bijection is in fact an isometry, we can slightly abuse notation and identify $\PP(\VV)$ with that subset. This subset is a cone, whose explicit description can be derived from part (3) of Proposition~\ref{prop:minkowski} with the help of Farkas' Lemma.

\begin{lemma}[Farkas' Lemma~\protect{\cite[Proposition 1.7]{Ziegler}}]
\label{lemma:Farkas}
Let $\alpha_1,\dots,\alpha_m\in (\R^n)^*$ and $b\in \R^m$. Then, the system of inequalities $\langle \alpha_i , x\rangle  \leq b_i, i \in \{1,\dots, m\}$, is infeasible if and only if there are $\lambda_1,\dots,\lambda_m\geq 0$ such that
\[
\sum_{i=1}^m \lambda_i\alpha_i  = 0, \qquad
\sum_{i=1}^m \lambda_i b_i < 0.
\]
\end{lemma}

\begin{theorem}
\label{thm:secondarycone}
Let $\VV= \{\alpha_1,\dots,\alpha_m\} \subset (\R^n)^*$ be a finite number of vectors positively spanning $(\R^n)^*$. 
The set $\PP(\VV)\subset \R^\VV$ of vectors $b$ that are support functions of polytopes with normals contained in $\VV$ is the cone defined by the following two types of inequalities:
\begin{itemize}
\item[(1)] For each nonnegative linear dependence $\sum_{i=1}^m \lambda_i \alpha_i =0$ among the vectors in $\VV$, (that is, linear dependence with all $\lambda_i\geq 0$) the inequality 
$\sum_{i=1}^m \lambda_i b_i \geq 0$.
\item[(2)] For each equality  $\sum_{i=1}^m \lambda_i \alpha_i = \alpha_j $ with all $\lambda_i\geq 0$,
the inequality $\sum_{i=1}^m \lambda_i b_i \geq b_j$.
\end{itemize} 
Moreover, only the linear dependences and combinations with minimal support (which are finitely many) are needed.
\end{theorem}

\begin{proof}
For each $b\in \R^\VV$, let us denote
\[
{\rm P}(b):= \bigcap_{i=1}^m\{x\in \R^n\,|\,\langle \alpha_i , x\rangle \leq b_i, i \in \{1,\dots,m\}\}.
\]
That the inequalities are satisfied for every $b\in \PP(\VV)$ follows from part (3) of Proposition~\ref{prop:minkowski}. Indeed, if
 we have
$\sum_{i=1}^m \lambda_i \alpha_i =0$ for some nonnegative $\lambda_i$'s then:
\[
0 = {\rm h}_{{\rm P}(b)}(0) = {\rm h}_{{\rm P}(b)}\Big(\sum_{i=1}^m \lambda_i \alpha_i\Big) \leq \sum_{i=1}^m \lambda_i {\rm h}_{{\rm P}(b)}(\alpha_i) = \sum_{i=1}^m \lambda_i b_i .
\]
Similarly, if $\sum_{i=1}^m \lambda_i \alpha_i =\alpha_j$ then
\[
b_j = {\rm h}_{{\rm P}(b)}(\alpha_j) = {\rm h}_{{\rm P}(b)}\Big(\sum_{i=1}^m \lambda_i \alpha_i\Big) \leq \sum_{i=1}^m \lambda_i {\rm h}_{{\rm P}(b)}(\alpha_i) = \sum_{i=1}^m \lambda_i b_i .
\]

Conversely, suppose that we have a certain $b\not\in \PP(\VV)$. That is, either the polytope ${\rm P}(b)$ is empty, or it is not empty but there is a $j$ for which $b_j$ is strictly larger than ${\rm h}_{{\rm P}(b)}(\alpha_j)$.

If the first thing happens, that is, if the system  $\bigcap_{i=1}^m\{x\in \R^n\,|\,\langle \alpha_i , x\rangle \leq b_i\}$ is infeasible, then Farkas' Lemma (Lemma~\ref{lemma:Farkas}) directly says that 
there is a linear dependence $\sum_{i=1}^m \lambda_i \alpha_i =0$ for which the inequality $\sum_{i=1}^m \lambda_i b_i \geq 0$ is violated.
If the second happens, then the following system is infeasible:
\[
\langle \alpha_i , x\rangle \leq b_i\ \forall i\in\{1,\dots,m\}, \quad \langle -\alpha_j , x\rangle \leq -b_j.
\]
Hence, Farkas' Lemma tells us that there are nonnegative $\lambda_1,\dots, \lambda_m, \mu$ such that 
$\sum_{i=1}^m \lambda_i \alpha_i - \mu \alpha_j = 0$ and $\sum_{i=1}^m \lambda_i b_i - \mu b_j  < 0.$ If $\mu$ happens to be zero, we are again violating one of the inequalities in part (1). If $\mu >0$ then we can assume without loss of generality that $\mu=1$ and we are violating the inequality $\sum_{i=1}^m \lambda_i b_i \geq b_j$ in part (2).

This proves that the inequalities of types (1) and (2) indeed define the cone $\PP(\VV)$. That the ones with minimal support are enough is easy. Suppose, for example, that we have an inequality of type (1) for which the corresponding nonnegative linear dependence does not have minimal support. Thus, there is another nonnegative dependence $\sum_{i=1}^m \mu_i \alpha_i =0$ with strictly smaller support. Let $t>0$ be the greatest number such that the vector $\lambda_i - t \mu_i\geq 0$ for every $i\in\{1,\dots,m\}$. Then, $\sum_{i=1}^m (\lambda_i - t \mu_i) \alpha_i =0$ is another nonnegative dependence with support strictly contained in that of the original one, and the inequalities for the dependences $\mu$ and $\alpha-t\mu$ imply the original inequality for $\lambda$. The proof for the inequalities of the second type is similar.
\end{proof}

\begin{remark}
We need to allow for $b_i$'s that produce polytopes
with fewer than $m$ facets (implying that some of the inequalities in \eqref{eq:polytope} become redundant) in order for $\PP(\VV)$ to be a closed cone. For example, the $b_i$'s defined by $b_i=\langle\alpha_i,p\rangle$ for a fixed $p\in \R^n$ produce, as mentioned above, the $0$-dimensional polytope $\{p\}$. The vectors $b$ of this form are a linear subspace of dimension $n$ of $\R^\VV$, namely the image of the linear map $\R^n\to \R^\VV$ whose matrix has $\alpha_1,\dots,\alpha_m$ as columns, and this linear subspace is the lineality space of $\PP(\VV)$. 

If $b,b'\in \PP(\VV)$ then the polytope produced by $b+b'$ is the Minkowski sum of the polytopes of $b$ and of $b'$. In particular, if one of them, say $b$ is in the lineality space, adding $b$ to a $b'$ corresponds to translating the polytope. That is, modding out the lineality space of $\PP(\VV)$ considered as a cone in $\R^\VV$ corresponds to modding out translations from $\PP(\VV)$ considered as a space of polytopes.
\end{remark}

\begin{example}
\label{exm:Hirzebruch}
Let $\VV$ be the set of rays in the Hirzebruch fan, that is, $\VV$ consists of
\[
\alpha_1=(0,-1), \alpha_2=(0,1), \alpha_3=(-1,0), \alpha_4=(1,k),
\]
where $k$ is a positive integer. See Figure~\ref{fig:Hirzebruch}.
We are looking at the polytopes of the form
\[
-x_2\leq b_1,\ x_2\leq b_2,\ -x_1 \leq b_3,\ x_1+kx_2\leq b_4,
\]
for different choices of $(b_1,b_2,b_3,b_4)$.

According to Theorem~\ref{thm:secondarycone} the inequalities that define the cone $\PP(\VV)$ 
come from the following linear equations among the $\alpha_i$:
\[
\alpha_1+\alpha_2=0,\quad
 k \alpha_1 + \alpha_3 + \alpha_4 = 0,\quad
\alpha_3 + \alpha_4 =   k \alpha_2,
\]
producing  according to Proposition~\ref{prop:minkowski} the inequalities
\begin{align*}
b_1+b_2  \geq  0,\quad
k b_1 + b_3 + b_4  \geq  0, \quad
b_3 + b_4  \geq  k b_2.
\end{align*}
The first inequality is easy to understand. Since we have $-b_1\leq x_2 \leq b_2$ among the inequalities defining our polytopes, we need $-b_1\leq b_2$ for the polytope to be non-empty. The second inequality is redundant since it is derived adding to the third $k$ times the first. Hence we have
\[
\PP(\VV) = \Big\{(b_1,b_2,b_3,b_4)\in \R^4 \,\Big|\, b_1+b_2\geq 0, b_3 + b_4 \geq k b_2\Big\}.
\]

Let us relate this  with the description of Hirzebruch trapezoids in Corollary~\ref{coro:2dim}.
Observe that, as said above, $\PP(\VV)$ has a \emph{lineality space} of dimension two, namely
\[
\Big\{(b_1,b_2,b_3,b_4)\in \R^4 \,\Big|\, b_1+b_2=0, k b_2 + b_3 + b_4 =0\Big\} = \left<  (0,0,1,-1), (1,-1,0,k) \right>.
\]
Modding out this subspace is equivalent to modding out translations in our space of polytopes. There are different ways of doing it, one of which is taking $b_1=b_2$ and $b_3=0$. This corresponds to having the origin as the mid-point of the left edge of the trapezoid, as we did in Corollary~\ref{coro:2dim}. Once this is done, the parameters $a$ and $b$ from that description correspond to our $b_4$ and $2b_1=2b_2$. Hence, the inequality $b_3+b_4 \geq kb_2$ that we have now corresponds exactly to the inequality $ 2a < bk$ that we had there, except for the fact that we are now including the equality case in the cone. This equality case corresponds to the trapezoid degenerating to a triangle. There is another degeneration associated to the inequality $b_1+b_2\geq 0$: when it is met with equality the two parallel edges of the trapezoid coincide so the trapezoid becomes a segment. When both inequalities become equalities we are in the lineality space, corresponding to the trapezoid degenerating to a point.
\end{example}

The cone $\PP(\VV)$ gets stratified into a finite number of subcones corresponding to the possible complete polytopal fans that we can obtain using the rays $\VV$.

\begin{definition}
Let $\NN$ be a complete  fan with set of ray generators contained in $\VV=\{\alpha_1,\dots,\alpha_m\}\subset (\R^n)^*$. We call \emph{secondary cone of $\NN$} the closure of 
\[
\Bigg\{b \in \R^\VV \,\Big|\, \text{the polytope $\bigcap_{i=1}^m\Big\{x \in \R^n \,|\, \langle \alpha_i,x\rangle \leq b_i \Big\}$ has normal fan $\NN$}\Bigg\} \subset \PP(\VV)\subset\R^\VV.
\]
\end{definition}

We show below that the stratification of $\PP(\VV)$ into the secondary cones of the finitely many possible fans with rays in $\VV$ is itself a  fan called the \emph{secondary fan} of $\VV$. This stratification is quite well understood in geometric combinatorics and is one of the main objects in the book \cite{triangbook} (see, e.g., Theorem 9.5.6 in \cite{triangbook}). 
Observe that in the language of \cite{triangbook} the complete fans with rays contained in $\VV$ are called \emph{polyhedral subdivisions} of $\VV$, with polytopal ones being called \emph{regular subdivisions} and simplicial ones called \emph{triangulations}.

This stratification also plays a role in toric geometry. If $\NN$ is a fan with set of rays $\VV$ then $\R^\VV$ can naturally be identified with the space of torus-invariant Weyl divisors, tensored with $\R$, on the toric variety $X_\NN$. Under this identification, the (open) secondary cone of $\NN$ is the cone spanned by \emph{ample Cartier} divisors in $X_\NN$ and its closure is spanned by \emph{nef} divisors. Because of this the secondary cone of $\NN$ is called the \emph{nef cone} of $X_\NN$ in \cite[Chapter 6]{CLS}.

In particular, the following statement is a rewriting of \cite[Corollary 5.2.7]{triangbook}. Its counterpart in toric geometry 
is \cite[Theorem 4.2.8 and Lemma 6.1.13]{CLS}:

\begin{theorem}
\label{thm:secondaryfan}
Let $\NN$ be a full-dimensional simplicial fan in $(\R^n)^*$ with set of ray generators $\VV=\{\alpha_1,\dots,\alpha_m\}$. Then the secondary cone of $\NN$
is a relatively open polyhedral cone in $\R^m$ defined by the following inequalities:
\begin{itemize}
\item[(1)] For each linear dependence $\sum_{i=1}^m \lambda_i \alpha_i =0$ with support contained in a cone of $\NN$, the equality $\sum_{i=1}^m \lambda_i b_i = 0$.
\item[(2)] For each linear combination $\sum_{i=1}^m \lambda_i \alpha_i$ with support contained in a cone of $\NN$ and producing as a result a certain $\alpha_j$, 
the inequality $\sum_{i=1}^m \lambda_i b_i < b_j$.
\end{itemize} 
Moreover, only the linear dependences and combinations with minimal support (which are finitely many) are needed.
The secondary cone is obtained by turning the strict inequalities in part {\rm (2)} into equalities.
\end{theorem}

\begin{proof}
This is just a rewriting of \cite[Corollary 5.2.7]{triangbook}.
\end{proof}

In Example~\ref{exm:Hirzebruch}  the secondary cone of the Hirzebruch trapezoid is the whole cone $\PP(\VV)$, since the Hirzebruch fan is the only simplicial fan using all the rays of $\VV$. We leave it to the reader to check that the systems of inequalities of Theorems~\ref{thm:secondarycone} and Theorems~\ref{thm:secondaryfan} are in this case equivalent (when in the second one the inequalities are taken nonnstrict).
In fact, the same is true for every two dimensional example, since a two-dimensional complete fan is fully characterized by the rays it uses.

\begin{remark}
It must be noted that our treatment of redundant inequalities is different from the one in \cite{triangbook}. 
Here we insist on those values to equal the support function of the  polytope $P(\VV,b)$, so in particular the $b_i$ of redundant inequalities are determined by the rest.
There, the values $b_i$ corresponding to redundant inequalities are not determined (as long as the corresponding inequalities stay redundant). This means that what we call secondary fan is actually a subfan of the one in \cite{triangbook}. For us only  the (polytopal) simplicial fans using \emph{all} the vectors from $\VV$ as rays give full-dimensional secondary cones, while in \cite{triangbook} all polytopal simplicial fans with rays contained in $\VV$ give full-dimensional cones.

This can be seen in that the analogue in \cite{triangbook} of our Theorem~\ref{thm:secondarycone} is  \cite[Theorem 4.1.39]{triangbook}, which has only the inequalities of type (1) and hence defines a larger cone.
\end{remark} 

The following summarizes the main results about secondary fans and secondary cones. We refer to \cite{triangbook} for the proofs.

\begin{theorem}[\protect{\cite[Proposition 5.2.9 and Theorem 5.2.11]{triangbook}}]
\label{thm:secondary}
Let $\VV$ be as above:
\begin{itemize}
\item[(1)] The secondary cones $\PP(\NN)$  for the different complete fans $\NN$ with rays contained in $\VV$ form a  fan covering $\PP(\VV)$. 
\item[(2)] The face poset of this fan is isomorphic to the refinement poset of fans. That is,  $\NN$ refines $\NN'$ if and only if the secondary cone of $\NN'$ is a face of that of $\NN$.
\item[(3)] The secondary cones of non-simplicial or non-polytopal fans are never full-dimensional.
\item[(4)] The secondary cone of a polytopal simplicial fan with $m'$ rays has dimension $m'$. In particular, full-dimensional secondary cones correspond bijectively to simplicial polytopal fans with set of rays exactly $\VV$.
\end{itemize}
\end{theorem}

Theorem~\ref{thm:hausdorff} implies that the $\|\cdot\|_\infty$ metric that $\PP(\VV)$ inherits from $\R^{\VV}$ is equivalent to the Hausdorff metric. More precisely, they are the same metric, modulo rescaling each coordinate $i$ of $\R^{\VV}$ by the norm of the corresponding $\alpha_i$. Hence:

\begin{theorem}
\label{thm:cells}
Let $\VV \subset (\R^n)^*$ be a configuration of $m$ non-zero vectors positively spanning $(\R^n)^*$. Then $\PP(\VV)$, considered as a  closed polyhedral cone in $\R^m$ with the description of Theorem~\ref{thm:secondarycone}, is isometric (modulo rescaling of coordinates) to  $\PP(\VV)$ considered as a subspace of $(\PP(n), \dist^H)$.
\end{theorem}

\begin{corollary}
\label{coro:cells}
Let $\NN$ be a complete simplicial fan in $(\R^n)^*$, with $m$ generators. Then, the subset of $\PP(n)$ consisting of polytopes whose normal fan equals $\NN$ is homeomorphic to $\R^m$. More precisely, it is isometric to the interior of a full-dimensional  polyhedral cone in $\R^m$ with a lineality space of dimension $n$, where $\R^m$ is considered 
with the $\|\cdot\|_\infty$ metric.
\end{corollary}

\subsubsection{Hausdorff dimension of the space of Delzant polytopes}

We can now answer Problems~5.3 and 5.5 from Fujita\--Ohashi~\cite{FuOh}.
For $m,n\in \N$
let $\DD_{\rm f}^m(n)$ and $\widetilde{\DD_{\rm f}^m(n)}$ denote the spaces of Delzant $n$-polytopes with $m$ facets, with and without $\AGL(n,\Z)$ equivalence.

\begin{theorem}\ 
\label{thm:dimension}
\begin{itemize}
\item[(1)] The Hausdorff dimension of both $\DD_{\rm f}^m(n)$ and $\widetilde{ \DD_{\rm f}^m(n)}$ is exactly $m$.
\item[(2)] The Hausdorff dimension of $\Dt$ and $\Dtt$ is $\infty$.
\end{itemize}
\end{theorem}

\begin{proof}
Both statements follow easily from the fact that the Hausdorff dimension of a countable union of sets equals the supremum of the Hausdorff dimensions of the individual sets.

For part (1), $\DD^m(n)$ is the union of the sets from Corollary \ref{coro:cells} for the different unimodular normal fans with $m$ rays. There is a countable amount of such fans, and the stratum for each fan has dimension $m$. In the case of $\widetilde{ \DD_{\rm f}^m(n)}$ the stratum of a certain fan $\NN$ is the quotient of its stratum in $\DD_{\rm f}^m(n)$ by the finite group of $\GL(n,\Z)$-automorphisms of $\NN$, so the quotient does not change dimension.

For part (2) simply observe that $\Dt$ and $\Dtt$ are the countable unions of $\DD_{\rm f}^m(n)$ and $\widetilde{ \DD_{\rm f}^m(n)}$ with $m\in N$.
\end{proof}

Fujita\--Ohashi~\cite[Problems~5.3, 5.5]{FuOh} ask about the dimension of the spaces $\DD_{\rm v}^{\ell}(n)$ and $\widetilde{ \DD_{\rm v}^{\ell}(n)}$ of Delzant $n$ polytopes with $\ell$ \emph{vertices}, rather than facets. In order to translate Theorem~\ref{thm:dimension} to this setting we only need to know what  the maximum number of facets of a Delzant $n$-polytope with $\ell$ vertices is.

\begin{lemma}[Lower bound theorem for simple polytopes~\cite{Barnette} (see also, e.g., \protect{\cite[Corollary 8.38]{Ziegler}})]
\label{lemma:LBT}
Let $P$ be a simple $n$-polytope with $m$ facets. Then, $P$ has at least 
\[
(n+1) + (m-n-1)(n-1) =   (m-n)(n-1) +2
\]
vertices, with equality if, and only if, the normal fan of $P$ is obtained from that of an $n$-simplex by $m-n-1$ stellar subdivisions at full-dimensional cones.
\end{lemma}

\begin{corollary}
\label{coro:dimension}
The Hausdorff dimension of both $\DD_{\rm v}^{\ell}(n)$ and $\widetilde{ \DD_{\rm v}^{\ell}(n)}$ is at most
\[
 n + \frac{\ell - 2}{n-1},
\]
with equality if and only if $\ell = 2 \pmod {n-1}$.
\end{corollary}

\begin{proof}
By the argument in the proof of Theorem~\ref{thm:dimension} the Hausdorff dimension equals the maximum number of facets in a Delzant polytope with $m$ vertices. By Lemma~\ref{lemma:LBT} this is at most $n + \frac{\ell - 2}{n-1}$ with equality only for iterated stellar subdivisions of a simplex, which necessarily have $\ell = 2 \pmod {n-1}$ and include iterated corner choppings of Delzant simplices.
\end{proof}

\begin{remark}
In dimension two $\DD_{\rm f}^m(2)=\DD_{\rm v}^{\ell}(2)$ is connected, except for the Hirzebruch case $m=4$ and the case $m=3$ (the latter is a point in $\widetilde{ \DD(2)}$ and a discrete space in $\DD(2)$). In dimension three or higher we expect every (or infinitely many) of the $\DD^m(n)$ to be not connected.
\end{remark}

\subsection{The fundamental group of $\Dt$} \label{more}
\label{sec:simply}

As mentioned above, the space $\PP(n)$ of all polytopes is contractible, since we can homotope the identity map to the constant map $\PP(n) \to P$ by Minkowski segments, where $P$ is any $n$-polytope. The same is true for $\PP_\Q(n)$ but not, at least not obviously, for  Delzant polytopes. In what follows we look at the fundamental group of $\Dt$:

Suppose that a path $\gamma:[0,1] \to \Dt$ has \emph{finite complexity}; that is, the number of rays (equivalently, the number of normal fans) that arise along the path is finite. Then, by Corollary~\ref{coro:desingular}, there is a polytope $Q$ whose normal fan refines that of $\gamma(t)$ for every $t$ and by Lemma~\ref{lemma:connected} we have the following well-defined homotopy:
\begin{align*}
\Gamma: [0,1]\times[0,1] &\longrightarrow \Dt\\
(t,s) &\longmapsto (1-s) \gamma(t) + s Q.
\end{align*}

\begin{corollary}
\label{coro:finite}
Loops of finite complexity in $\Dt$ are null-homotopic.
\end{corollary}

\begin{proof}
If $\gamma$ is a loop of finite complexity with $\gamma(0)=\gamma(1)=P$, the homotopy $\Gamma$ above shows that $\gamma$ is homotopic to the Minkowski path $P Q Q  P$, which is obviously null-homotopic.
\end{proof}

A similar argument can be applied to the following class of paths.

\begin{definition}
\label{defi:refining}
We say that a path $\gamma:[0,1] \to \Dt$  is \emph{locally refining} if 
for every $t_0\in [0,1]$ there is an $\epsilon$ such that the fan of every polytope $\gamma(t)$ with $|t-t_0|\leq \epsilon$ refines that of $\gamma(t_0)$. 
\end{definition}

\begin{theorem}
\label{thm:refining}
Locally refining loops in $\Dt$ are null-homotopic.
\end{theorem}

\begin{proof}
By compactness, there are $0=t_0<t_1< \dots < t_k=1$ such that each interval $[t_{i-1}, t_i]$, $i\in\{1,\dots, k\}$, contains a point $r_i$ such that the normal fan of $\gamma(t)$ refines that of $\gamma(r_i)$, for every $t\in [t_{i-1}, t_i]$. This makes the following homotopy to be well-defined in $\Dt$:
\begin{align*}
\Gamma: [t_{i-1}, t_i]\times[0,1] &\longrightarrow \Dt\\
(t,s) &\longmapsto (1-s) \gamma(t) + s \gamma(r_i).
\end{align*}
Hence, each subpath $\gamma|_{[t_{i-1}, t_i]}$ is homotopic to the Minkowski path $\gamma(t_{i-1})  \gamma(r_i)  \gamma(t_i)$ and  the whole loop $\gamma$ is homotopic to a loop of finite complexity, which is null-homotopic by Corollary~\ref{coro:finite}.
\end{proof}

Locally refining loops are a more general class than loops of bounded complexity, but they do not include all loops:

\begin{lemma}
Every loop of finite complexity is locally refining.
\end{lemma}

\begin{proof}
Suppose that a loop $\gamma$ has finite complexity, so that along the loop only a finite number of normal fans, say $\NN_1,\dots,\NN_k$ arise.
Observe that this does not imply that the fan  changes only at a finite number of values of $t$, since we can alternate infinitely many times between two or more fans.

Suppose, to seek a contradiction, that $\gamma$ is not locally refining. That is, 
there is a $t_0\in [0,1]$ such that in every interval $(t_0-\varepsilon, t_0+\varepsilon)$ a fan that does not refine that of $\gamma(t_0)$ arises. Call the latter $\NN_0$. By finiteness of the number of fans, there is a particular fan $\NN$ not refining $\NN_0$ and a sequence $(t_i)_{i\in \N}$ converging to $t_0$ such that the normal fan of $\gamma(t_i)$ equals $\NN$ for every $t_i$. Let $\VV$ be the set of primitive generators of rays in $\NN$. Then, the secondary cone of $\NN$ in the cone $\PP(\VV)$ of Section~\ref{sec:stratification}, which is closed in $\PP(n)$ since it is complete, contains $\gamma(t_0)$ in its boundary. This contradicts the fact that the boundary of the secondary cone of $\NN$ consists of the secondary cones of fans refined by $\NN$, while $\NN_0$ is not refined by $\NN$.
\end{proof}

\begin{example}[A loop that is not locally refining] \label{e:notlocref}
\label{exm:nonlocal}
Let 
\[
P=\Conv\big\{(0,0,0), (1,0,0), (0,1,0), (0,0,1)\big\}
\] 
be the standard tetrahedron in $\R^3$. Consider the following perturbations of it, with $k\in \N$:
\[
P_k = \Conv\big\{(0,0,0), (1,1/(k+1),-1/k), (0,1,0), (0,0,1)\big\}
\]
Since the Hausdorff distance between two polytopes is attained at a vertex of one of them, and $P$ and $P_k$ have all but one vertices in common, we have that
\[
\dist^H(P,P_k)\leq \|(1,0,0) - (1,1/(k+1),-1/k\| < 2/k.
\]

Of course, $P_k$ is not a Delzant polytope. The crucial property that we need is that its normal fan contains the 2-dimensional cone $C_k$ generated by  the vectors
\[
(-1,0,-k) \text{ and } (1,-k-1,0).
\]
This cone has multiplicity one for every $k$ since it is a face of a unimodular cone: 
\[
\left| \begin{matrix} -1&0&-k \\ 1&-k-1&0 \\ 1&1&k+1 \end{matrix}\right| = 1
\]
The relative interior of $C_k$ intersects the relative interior of 
$\Cone\{ (0,-1,0), (0,0,-1)\}$ at the ray generated by $(0,-k-1,-k)$. In particular, no  fan containing $C_k$ as a cone refines the normal fan of the standard tetrahedron $P$.

Now, let $Q_k$ be a Delzant polytope whose normal fan refines that of $P_k$, obtained by the method in Theorem~\ref{thm:desingular}; that is, by a sequence of stellar subdivisions at rays lying in the relative interior of cones of multiplicity greater than one. Then, the polytope $Q_k$ still contains the cone $C_k$ in its normal fan, so the normal fan of $Q_k$ does not refine that of the standard simplex $P$.
We also assume that 
\[
\dist^H(P,Q_k) < 2/k,
\]
which is no loss of generality since $\dist^H(P,P_k) < 2/k$ and the Minkowski segment between any $Q_k$ and $P_k$ gives us polytopes with the same normal fan as $Q_k$ and arbitrarily close to $P_k$. 

Finally, let $R_k$ be an arbitrary Delzant polytope whose normal fan refines that of both $Q_k$ and $Q_{k-1}$, and assumed to still be at distance at most $2/k$ from both $P$.

We are now ready to define the path $\gamma$. Restricted to each interval $[1/k,1/(k-1)]$ we define $\gamma$ as the Minkowski path going from $Q_k$ to $R_k$ and then from $R_k$ to $Q_{k-1}$. The concatenation of this infinitely many paths defines a path $\gamma:(0,1] \to \DD(3)$, and the fact that all the polytopes in the $k$-th interval are at distance $\leq 2/k$ from $P$ implies that we can extend this path to $[0,1]$ by simply defining $\gamma(0)=P$. The path so constructed is not locally refining since in every interval $[0,\varepsilon)$ infinitely many of the normal fans of $Q_k$ appear, and none of them refines that of $P=\gamma(0)$. Concatenating $\gamma$ with any Minkowski path from $\gamma(1)$ to $P$ we get a loop that is not locally refining.
\end{example}

\begin{remark}
The path $\gamma$ constructed in Example~\ref{exm:nonlocal} shows that Theorem~\ref{thm:refining} is not enough to conclude the simple connectedness of $\Dt$ but the path itself is null-homotopic by the following argument. Let $R'_k$ be a fan refining that of $P$ and of $R_k$, and still assumed to be at distance strictly less than $2/k$ of $P$. We can homotope $\gamma$ to the path $\gamma'$ that uses $R'_k$ instead of $R_k$ in each interval $[1/k,1/(k-1)]$, and then homotope $\gamma'$ to a path $\gamma''$; this substitutes the polytope $Q_k$ that we now have between $R'_k$ and $R'_{k+1}$ by a polytope $Q'_k$ whose fan refines that of $R'_k$ and $R'_{k+1}$, hence that of $P$. This produces a path homotopic to $\gamma$ but in which all normal fans refine that of $P$.
\end{remark}

However, for $n=2$ we have the following result,  stronger than \cite[Theorem 5.1.5]{FuKiMi}:

\begin{theorem}
\label{thm:polygon-converge}
If a sequence $(P_k)_{k\in \N}$ of Delzant polygons converges to a Delzant polygon $P$ then there is a $k_0$ such that the normal fan of $P_k$ refines that of $P$ for every $k\geq k_0$.
\end{theorem}

\begin{proof}
Since $\NN(P)$ has finitely many rays it is enough to concentrate on a particular normal ray $\alpha_0$ of $P$ and show that there is a $k_0$ such that every $\NN(P_k)$ with $k\geq k_0$ contains that ray. Without loss of generality, let $\alpha_0=(1,0)$. 

By  \cite[Corollary 5.1.4]{FuKiMi} there must be a ray $\alpha_k$ in each $\NN(P_k)$ such that the sequence $(\alpha_k)_k=(a_k,b_k)_k$ converges  to $\alpha_0$. (Here convergence is meant as rays,  not as generators; that is, their slopes $b_k/a_k$ converge to zero). This is because $P_k$ needs to have an edge, or a path of them, converging to the edge of $P$ having $\alpha_0=(1,0)$ as normal vector.

This $\alpha_k$ may have positive or negative slope, but we can take a subsequence where the slopes have the same sign, and assume it positive. That is, the ray has $a_k>0$, $b_k\geq 0$ and $b_k/a_k$ converging to zero. Moreover, there is no loss of generality in assuming that $\alpha_k$ is the last ray in $\NN(P_k)$, in clockwise order,  with these properties, since any ray between $\alpha_k$ and the ray $(1,0)$ would have a smaller slope and would still give a sequence with slopes converging to zero.

Then, the ray of $\NN(P_k)$ after $\alpha_k$ in clockwise direction is of the form $\beta_k=(-c_k, -d_k)$ with $d_k>0$. For the corresponding cone to be unimodular we need $a_k d_k - c_kb_k =  1$, which implies that  $c_k$ is also positive and that $c_k/d_k$ converges to infinity too. That is, the rays $\beta_k$ lie in the negative quadrant and have slopes converging to zero.

Now, let $\beta=(e,-1)$ be the next ray after $(1,0)$ in the normal fan of $P$, in clockwise order ($e$ can be positive or negative). 
The contradiction is that the polygons $P_k$ need to also have rays converging to $\beta$, but this is impossible by the presence of the cone
$\Cone\{\alpha_k,\beta_k\}$ in $\NN(P_k)$.
\end{proof}

\begin{corollary}
\label{coro:polygon-converge}
All paths in $\DD(2)$ are locally refining.
\end{corollary}

\begin{proof}
By Theorem~\ref{thm:polygon-converge}, each Delzant polygon $P$ has a neighborhood in $\DD(2)$ such that the normal fan of every polygon in the neighborhood refines that of $P$.
\end{proof}

\begin{corollary}
\label{coro:simply2}
$\DD(2)$ is simply connected.
\end{corollary}

\begin{proof}
Every loop in $\DD(2)$ is locally refining by the Corollary~\ref{coro:polygon-converge}, and locally refining loops are null-homotopic by Theorem~\ref{thm:refining}.
\end{proof}

\begin{question}
\label{q:simply-connected}
Is the fundamental group of $\Dt$ trivial?
\end{question}

\subsection{The {\rm CW} topology on $\Dt$.}
\label{sec:CW-contractible}

Since $\PP(n)$ is the union of the subsets $\PP(\NN)$ for the different complete fans $\NN$ and since each $\PP(\NN)$ is isometric to a finite dimensional polyhedral cone, the stratification by normal equivalence is a cell decomposition of $\PP(n)$ into the uncountably many strata $\PP(\NN)$ for the possible  fans.
The closure of each cell is isometric to a closed polyhedral cone in some $\R^m$ with its face structure being the poset structure of the finitely many fans refined by $\NN$. In order to make these closures compact, we only need to take the quotient of $\PP(n)$ by translations and positive scaling.%
Indeed, the quotient by translations removes the common lineality space of all closed cells, turning them into pointed cones, and the quotient by positive scaling makes these pointed cones compact; indeed, homeomorphic to polytopes, and preserving their face structure.
That is:

\begin{proposition}
\label{prop:CW}
The quotient of $\PP(n)$ by translations and positive scaling has a natural structure of regular {\rm CW} complex in which closed cells of dimension $m$ correspond to secondary cones of dimension $m+n+1$ and the face poset structure of the {\rm CW} complex is that of refinement of normal fans.

In this structure $\PP_\Q(n)$ is a subcomplex with countably many cells, and $\Dt$ is a union of open cells, but not a subcomplex.
\end{proposition}

We call the topology induced on any of $\PP(n)$, $\PP_\Q(n)$ and $\Dt$ spaces by the {\rm CW} structure their \emph{{\rm CW} topology}. Observe that a priori this is defined on a quotient space, but it can be lifted to the original space via the following considerations:

\begin{itemize}

\item One canonical way to choose a representative from each class in the quotient is to translate and dilate every polytope other than a single point (a $0$-polytope) so that its circumball is the unit ball. The \emph{circumball} or smallest enclosing ball exists and is unique for every bounded subset of $\R^n$~\cite[p.~49]{Gruber}. This shows that the quotient space is also a subspace.

\item Let $ \widetilde{\PP(n)}$ denote the quotient of $\PP(n) \setminus \{ \text{$0$-polytopes}\}$ by translations and dilations. 
The following map is a bijection:
\begin{align*}
\PP(n) \setminus \Big\{ \text{$0$-polytopes}\Big\}& \longrightarrow \widetilde{\PP(n)} \times \R^n \times (0,\infty) \\
P \qquad &\longmapsto \ (\widetilde P\ , \ {\rm c}(P)\ ,\ {\rm r}(P))
\end{align*}
where $\widetilde P$, ${\rm c}(P)$ and ${\rm r}(P)$ denote, respectively, the class of $P$, the center, and the radius of the circumball of $P$.

Via this bijection, we can lift the {\rm CW} topology of $\widetilde{\PP(n)}$ to ${\PP(n)}$, and to its subspaces $\PP_\Q(n)$ and $\Dt$. Since $\R^n \times (0,\infty)$ is contractible, the homotopy type of $\Dt$ and its quotient is the same.
\end{itemize}

Regarded from this perspective, the fans constructed in Theorem~\ref{thm:isolated} correspond to open cells that are in $\Dt$ but with empty boundary in $\Dt$. Also, this provides another way of expressing  why connectivity by blow-ups/blow-downs as we have for $n=2$ is nicer than the one we have for $n\geq 3$: for $n=2$ there are paths that move only through cells of consecutive dimensions.

However, it must be noted that the {\rm CW} topology obtained in this way is not the metric topology of $\PP(n)$ that we are considering in the paper, coming from $\dist^H$ or $\dist^V$. In fact, we have the following:

\begin{lemma}
The {\rm CW} topology on $\PP(n)$ (or on $\PP_\Q(n)$, or on $\DD(n)$) is not metrizable. Hence, it does not coincide with the topology coming from $\dist^V$ o $\dist^H$.
\end{lemma}

\begin{proof}
A metrizable {\rm CW} complex must be locally finite, i.e., every point must lie only in finitely many closed cells. In our case every polytope $P$ lies in the infinitely many cells corresponding to fans that refine $\NN(P)$.
\end{proof}

\begin{proposition}
\label{prop:finer}
The {\rm CW} topology on $\PP(n)$ (or on $\PP_\Q(n)$, or on $\DD(n)$) is strictly finer than the metric topology coming from $\dist^V$ o $\dist^H$.
\end{proposition}

\begin{proof}
In the {\rm CW} complex a set is, by definition, open (respectively,  closed) if and only if its intersection with each closed cell is open (respectively, closed).
Since the closed cells of the {\rm CW} complex are closed also in the metric topology, the {\rm CW} topology is finer. It is strictly finer since, for example, $\PP_\Q(n)$ is closed in $\PP(n)$ with respect to the {\rm CW} topology (as is any union of closed cells in a {\rm CW} complex) while it is dense in the metric topology (cf. Theorem~\ref{thm:connected}). 
\end{proof}

However, for every finite, or even locally finite, subcomplex the {\rm CW} topology and the metric topology in $\PP(n)$ coincide.

As another proof that the topologies are different, Example~\ref{exm:nonlocal} shows a path $\gamma$ in $\DD(3)$ with respect to the metric topology in which every neighborhood of $P_0=\gamma(0)$ contains polytopes whose fan does not refine that of $\gamma_0$. In the {\rm CW} topology, the union of cells corresponding to polytopes that refine $\NN(P)$ are, by definition, a neighborhood of $P$. More strongly, in the CW topology all paths have finite complexity:

\begin{lemma}
\label{lemma:CW-continuous}
Let $n\in \N$ and let $X$ be a compact topological space. A map $f: X \to \Dt$ is continuous with respect to the {\rm CW} topology
in $\Dt$ if and only if:
\begin{itemize}
\item[(1)] It is continuous with respect to the metric topology, and
\item [(2)] $f(X)$ intersects only finitely many cells.
\end{itemize}
\end{lemma}

\begin{proof}
Necessity of (1) follows from Proposition \ref{prop:finer} and necessity of (2) from the fact that $f(X)$ is compact and every compact subspace in a {\rm CW} complex is contained in a finite subcomplex \cite[Proposition A.1]{Hatcher}.

In order to prove sufficiency, suppose that $f$ intersects only $k$ cells, corresponding to the unimodular fans $\Sigma_1,\dots, \Sigma_k$. Let $\Sigma$ be a common unimodular refinement of them and let $\VV$ the finite set of rays used in $\Sigma$. Then $f$ can be considered as a map $X \to f(X) \subset \PP(\VV)$ and in $ \PP(\VV)$ the {\rm CW} topology and the metric topology coincide.
\end{proof}

Recall that a space with all homotopy groups trivial is called \emph{weakly contractible}.

\begin{corollary} 
\label{coro:homotopy}
$\Dt$ with its {\rm CW} topology is weakly contractible.
\end{corollary}

\begin{proof}
Let $f: \S^k \to \Dt$ be a continuous map from the $k$-sphere $\S$ to $\Dt$. We need to show that it is homotopic to a constant map.
By \cite[p. 346]{Hatcher} to show that a homotopy group is trivial we can forget basepoints.
By Lemma~\ref{lemma:CW-continuous} there is a Delzant polytope $Q$ whose normal fan refines $f(x)$ for every $x \in \S^k$.
We consider the following homotopy, similar to the one in Corollary~\ref{coro:finite} and Theorem~\ref{thm:refining}:
\begin{align*}
\Gamma: \S^k\times[0,1] &\longrightarrow \Dt\\
(x,s) \quad &\longmapsto (1-s) f(x) + s Q.
\end{align*}
This map is continuous with respect to the metric topology and its image intersects only the cells which $f(\S^k)$ intersects plus the cell corresponding to $Q$.
\end{proof}

\begin{remark}
A \emph{weak homotopy equivalence} between two topological spaces is a continuous map inducing isomorphisms in all homotopy groups. 
Whitehead's Theorem~\cite[Theorem~4.5]{Hatcher} says that a weak homotopy equivalence between CW complexes is in fact a homotopy equivalence, so weakly contractible  {\rm CW} complexes are contractible. We cannot apply this directly to $\Dt$ since it is a union of open cells, but not a subcomplex, of $\PP(n)$. 
\end{remark}

\section{Translation to symplectic toric manifolds} \label{sec:symplectic}

In this section we translate the main result from Sections~\ref{sec:dp} and~\ref{sec:dp2} into the language of symplectic toric manifolds.

\subsection{Symplectic toric manifolds: basic theory}
\label{sec:symplectic-intro}

For an introduction to symplectic geometry and Hamiltonian torus actions we refer to 
Cannas da Silva~\cite{AC},  Guillemin\--Sjamaar~\cite{GSj05} and
Pelayo~\cite{PES}. The origins of symplectic geometry can be traced back to classical mechanics, see 
Abraham\--Marsden~\cite{AbMa}, De Le\'on\--Rodrigues~\cite{LeRo} 
and Marsden\--Ratiu~\cite{MaRa} for texts with an emphasis on this point of view.

\subsubsection{Definition of symplectic toric manifold}\label{stm}

We start reviewing the concept of symplectic manifolds, which is the mathematical model of the phase space of physics.

\begin{definition}
A \emph{symplectic manifold} $(M,\omega)$ is a pair 
consisting of a smooth manifold $M$ and a \emph{symplectic 
form} $\omega$, i.e., a non-degenerate closed 2-form on $M$. 
We denote by 
$\operatorname{Sympl}(M)$ the group of symplectomorphisms (i.e. symplectic diffeomorphisms) of $M$.
\end{definition}

Throughout this section we wtrite 
\[
\mathbb{T}^n:=(\T^1)^n=(\R/\Z)^n
\] 
for the $n$-dimensional standard torus. Denote by $\mathfrak{t}$ the Lie algebra $\mbox{Lie}(\T^n)$ 
of $\T^n$ and by $\mathfrak{t}^*$ the dual of $\mathfrak{t}$.   Let ${\rm exp} \colon \mathfrak{t} \to \T^n$ be the 
exponential map of Lie theory.

Let  $\T^n \times M \to M, (t,m) \mapsto t \cdot m$ be a symplectic action (i.e. an action by symplectomorphisms) of the $n$\--dimensional torus $\T^n$ on $(M,\omega)$. Let $X \in \mathfrak{t}$ . Then $X$ generates a 
smooth vector field $X_M$ on $M$, called the 
\emph{infinitesimal generator of the action}, defined
by 
$$
X_M(m):= \left.\frac{{\rm d}}{{\rm d}t}\right|_{t=0}
{\rm exp}(tX)\cdot m.
$$ 
We write
$\iota_{X_M} \omega : = \omega(X_M, \cdot) \in \Omega^1(M)$
for the contraction $1$-form.

\begin{definition}
Let $\left\langle \cdot , \cdot
\right\rangle : \mathfrak{t}^\ast \times \mathfrak{t}
\rightarrow \mathbb{R}$ be the duality pairing.%
\footnote{In previous sections we used the same notation for the pairing in general real vector spaces or in $\R^n$.} 
The  $\T^n$-action on $(M,\omega)$  is said to be 
\emph{Hamiltonian} if there exists a smooth
$\mathbb{T}^n$-invariant map $\mu\colon M \to \mathfrak{t}^*$, 
called the \emph{momentum map} or \emph{moment map}, such that for
all $X \in \mathfrak{t}$ we have 
$
\iota_{X_M}
\omega ={\rm d} \langle \mu,
X \rangle
$.
\end{definition}

In order to simplify the presentation we are going to assume that the momentum
map takes values in $\R^n$, where $n$ is the dimension of $\mathfrak{t}^*$. This will allow us to
work and reason with Delzant polytopes etc in $\R^n$, as we did in the previous sections, instead of with the abstract space $\mathfrak{t}^*$.
This is a standard simplification (which on the other hand represents no loss of generality) 
done via the following non-canonical procedure. Let us choose an epimorphism 
$E \colon \mathbb{R} \to \mathbb{T}^1$, for instance, 
$x \mapsto e^{2\pi\, \sqrt{-1}x}$. This Lie group epimorphism 
has discrete center $\mathbb{Z}$ and the inverse of the  
corresponding Lie algebra isomorphism is given by
$\mbox{Lie}(\mathbb{T}^1) \ni \frac{\partial}{\partial x} 
\mapsto \frac{1}{2\pi}\in \mathbb{R}$. Therefore, for 
$\mathbb{T}^n = \left(\mathbb{T}^1\right)^n$, we obtain the following 
 isomorphism of Lie algebras
\begin{align*}
\mbox{Lie}(\mathbb{T}^n) =
\mathfrak{t} &\longrightarrow \R^n\\
\frac{\partial}{\partial x_k} &\longmapsto 
\frac{1}{2\pi} \, {\rm e}_k,
\end{align*}
where  ${\rm e}_k$ is the $k$th standard basis vector in $\mathbb{R}^n$. Choosing an inner product  on $\mathfrak{t}$, we obtain an 
isomorphism $\mathfrak{t} \to \mathfrak{t}^*$, and hence 
taking its inverse and composing it with the isomorphism
$\mathfrak{t} \to \mathbb{R}^n$ described above, we get an 
isomorphism $\mathcal{I}: \mathfrak{t}^* \to \mathbb{R}^n$. 
In this way, we obtain a momentum map $$\mu=\mu_{\mathcal{I}} \colon M \to \mathbb{R}^n.$$

\begin{definition}
\label{def:symplectic}
A \emph{symplectic 
toric manifold} is quadruple of the form $$(M, \om, \T^n, \mu),$$ where $(M, \om)$ 
is a compact connected symplectic manifold of dimension $2n$ endowed with 
an effective Hamiltonian action of an 
$n$\--dimensional torus $\T^n$ admitting a momentum map $\mu: M \to 
\mathfrak{t}^*$.  Via the (non-canonical) isomorphism $\mathfrak{t}^*\cong \R^n$
the map $\mu$  gives rise to a map
$M \to \mathbb{R}^n$  which, to simplify notation, 
we also denote by $\mu \colon M \to \mathbb{R}^n$.
 \end{definition}
 
 Often symplectic toric manifolds are referred to as (compact connected) \emph{toric integrable systems}. Essentially a toric integrable system
 on a $2n$\--dimensional symplectic manifold $(M,\omega)$  is a smooth map $\mu \colon (M,\omega) \to \R^n$ which is
 an integrable system in the usual sense (see for instance~\cite{PeVN-BAMS11} for the basic definitions and properties) in
 which all components $\mu_i \colon M \to \R$  generate periodic flows of the same period. Usually one does not
 require the manifold to be compact when speaking of toric integrable systems, unlike in the case of symplectic toric manifolds,
 but other than that they are the same object.

Symplectic toric manifolds also have the structure of a toric variety; see Delzant~\cite{De1988} for this connection and Duistermaat\--Pelayo~\cite{DuPe2009} for explicit 
coordinate formulas explaining it.  For surveys concerning symplectic geometry
from the point of view of Hamiltonian torus actions and integrable systems, including symplectic toric manifolds, we refer to \cite{PeVN-BAMS11, PES}. 

\subsubsection{Isomorphisms and weak isomorphisms of symplectic toric manifolds}\label{isomorphism}

\begin{definition} \label{def:iso}
Let $(M,\omega,\T^n,\mu)$ and $(M',\omega',\T^n,\mu')$ be  $2n$\--dimensional
symplectic toric manifolds endowed with effective symplectic 
actions $\rho\colon \T^n\to \symp (M,\omega)$ and $\rho'\colon 
\T^n\to \symp (M',\omega')$. We say that  $(M,\omega,\T^n,\mu)$ 
and $(M',\omega',\T^n,\mu')$ are:
\begin{itemize}
\item
\emph{isomorphic} if there exists an equivariant symplectomorphism 
$\varphi\colon M\to M'$ such that $\mu'\circ \varphi=\mu$.
\item
 \emph{weakly isomorphic} if there 
exists an automorphism  $h\colon \T^n\to \T^n$ and 
an $h$\--equivariant symplectomorphism $\varphi\colon M\to M'$, 
i.e., the following diagram commutes:
\begin{eqnarray}\label{comm action2} \nonumber
\xymatrix{
\T^n\times M \;\; \ar[r]^{\;\;\rho^*} \ar[d]_{(h,\varphi)} & M 
\ar[d]_{\varphi} \\
\T^n \times M' \;\;\ar[r]^{\;\;\rho'^*} &  M', } 
\end{eqnarray}
where $\rho^\ast(x,m): = \rho(x)(m)$ for every $x\in\mathbb{T}^n$, 
$m \in M$ and $\rho'^*$ is defined analogously.
\end{itemize}
\end{definition}

Since the automorphism group of  $\T^n$ 
is  the general linear group $\GL(n,\Z)$ we have that $h$ is represented 
by a matrix $A\in \GL(n,\Z)$.

\subsubsection{Atiyah\--Guillemin\--Sternberg Convexity Theorem} \label{sec:ags}

Let $(M,\omega)$ be a symplectic manifold of dimension $2n$ and let $(\T^1)^k$ be a torus of
any dimension $k\leq n$ acting  on $M$ Hamiltonianly with momentum map $\mu \colon M \to \R^k$.
A seminal result proven independently by Atiyah~\cite{At82} and Guillemin\--Sternberg~\cite{GS82} says that the
momentum map image
\[
\mu(M) \subset \R^k
\]
is equal to the polytope $P \subset \R^k$ obtained by taking the convex hull of
the images of the fixed point set of the $(\T^1)^k$\--action on $M$. In our case, since
we are interested in symplectic toric manifolds, the  relevant case is $k=n$.  So
if $(M,\omega, \T^n,\mu)$ is a $2n$\--dimensional symplectic toric manifold then
$\mu(M)$ is a  polytope in $\R^n$.

This general result by Atiyah and Guillemin\--Sternberg has some important precedents  in the work of Horn~\cite{HORN}, 
Kostant~\cite{Kostant} and Schur~\cite{SCHUR}.

\subsubsection{Definition of the spaces of symplectic toric manifolds}

\begin{definition}[Sections 1.2 and 1.3 in~\cite{PPRS14}]
Assuming the previous conventions for
$\T^n$ and that the identification $\mathcal{I} \colon
\mathfrak{t}^* \to \mathbb{R}^n$ is \emph{fixed}, we 
define:
\begin{itemize}
\item
$
\Mo:=\Mo^{\mathcal{I}}
$
as the set of equivalence classes of  $2n$\--dimensional symplectic toric manifolds under the equivalence relation given by
 isomorphisms\footnote{If we choose a different identification 
$\mathcal{I}' \colon \mathfrak{t}^* \to \mathbb{R}^n$,
then $\Mo^{\mathcal{I}'}$ and $\mathcal{M}_{\T^n}$ are in general different as sets. 
However they are in bijective correspondence 
by a bijection  which preserves the structures with which we are concerned in the present paper.}.
 \item
$\Mt$
as the set of equivalence classes of  $2n$\--dimensional symplectic toric manifolds under the equivalence relation given by
weak  isomorphisms\footnote{Two weakly isomorphic toric manifolds 
are isomorphic if $h$ is the identity and 
$\mu' \circ \varphi =\mu$.}. 
\end{itemize}
\end{definition}

To simplify the notations we often write
$(M,\omega,\T^n,\mu)$ identifying the representative with
its equivalence class $\left[(M,\omega,\T^n,\mu) \right]$
in $\mathcal{M}(2n)$.

\subsubsection{Delzant correspondence: from symplectic toric manifolds to Delzant polytopes}

We saw in Section~\ref{sec:ags} that if $(M,\omega, \T^n,\mu)$ is a $2n$\--dimensional symplectic toric manifold then
$\mu(M)$ is a  polytope in $\R^n$. 

A famous theorem by Delzant showed that the polytopes that arise this way are exactly the Delzant polytopes 
(Definition~\ref{delpol})
and that the Delzant polytope of a symplectic toric manifold is a complete invariant of its isomorphism type. This was complemented by Karshon\--Kessler\--Pinsounnault~\cite[Proposition 2.3 (2)]{KKP} who gave the analogous result for weak isomorphisms:

\begin{theorem}[Delzant~\cite{De1988}, Karshon\--Kessler\--Pinsounnault~\cite{KKP}]
\label{thm:delzant}
Let $\Dt$ and $\Dtt$ be the space of Delzant $n$-polytopes and its quotient by $\AGL(n,\Z)$-equivalence, as defined in Section~\ref{sec:metric}. Then
\begin{itemize}
\item The map
\begin{align}
\label{eq:map}
 \Mo &\longrightarrow \Dt.\\
\left[(M,\omega,\T^n,\mu) \right]   &\longmapsto \mu(M).
\nonumber
\end{align}
is a bijection between $\Mo$ and $\Dt$.

\item This bijection descends to a bijection
\begin{equation}\label{iso moduli}
\Mt \longrightarrow \Dtt.
\end{equation}

\end{itemize}
\end{theorem}

\begin{example}[See Corollary~\ref{coro:2dim} and Figure~\ref{fig:Hirzebruch}] 
\[ 
P_\lambda = \left\{ \left.(x_1,x_2) \in \mathbb{R}^2 
\ \right| \
 x_1 \geq 0 ,\; x_2 \geq 0 ,\,\,\, x_1 + x_2 \leq \lambda 
 \right\}, 
\]
is the image under this map of the standard $\T^2$ action on $
\mathbb{CP}^2$ endowed
with the Fubini-Study symplectic form multiplied by $\lambda$. The \emph{Hirzebruch trapezoid}, 
\[ 
{\rm H}_{a,b,k} = \left\{ (x_1,x_2)\in\mathbb{R}^2 \ \left| \
   -\frac{b}{2} \leq x_2 \leq \frac{b}{2}, \right.\,\,\, 0 \leq x_1 
   \leq a - kx_2 \right\}, 
\]
is the image under this map of the standard toric action on a 
Hirzebruch surface.  
\end{example}

\subsection{Geometry and topology of the space of symplectic toric manifolds}
\label{sec:space-symplectic}

\subsubsection{Symplectic interpretation of Theorems~\ref{thm:dim3} and \ref{thm:isolated}}
\label{sec:dim3-symplectic}
Theorem~\ref{thm:dim3} and Theorem~\ref{thm:isolated}  imply that, unlike in dimension $4$, there is no collection of ``toric minimal models" for  symplectic toric manifolds in dimension $6$ or higher.
We refer to Delzant~\cite{De1988} and Duistermaat-Pelayo~\cite{DuPe2009} for the relation between a symplectic toric manifold and the underlying toric variety.

\begin{theorem}
\label{thm:dim3-symplectic}
Let $(M,\omega)$ be a 6-dimensional symplectic toric manifold. Then there exists another $6$-dimensional symplectic toric manifold $(M',\omega')$ such that:
\begin{itemize}
\item there is a toric equivariant epimorphism   between the underlying toric varieties $M \twoheadrightarrow M'$,
\item $(M', \omega')$ is not equivariantly symplectomorphic to the toric equivariant blow up of any other $6$-dimensional symplectic toric manifold.
\end{itemize}
\end{theorem}

\begin{corollary}
\label{coro:dim3-symplectic}
Given any collection $\mathcal{F}$ of $6$-dimensional symplectic toric manifolds there exists another $6$-dimensional symplectic toric manifold which cannot be obtained by a finite sequence of toric equivariant symplectic blow ups from an element of $\mathcal{F}$. 
\end{corollary}

\begin{theorem}
\label{thm:isolated-symplectic}
Let $n\geq 3$. There is an infinite family $(M_k,\omega_k)_{k\in \N}$ of pairwise non-homeomorphic $2n$-dimensional symplectic toric manifolds with the property that any  toric equivariant epimorphism  to another $2n$-dimensional symplectic toric manifold $M_k \twoheadrightarrow M$ is  a homeomorphism.
\end{theorem}

\subsubsection{Metric and topology on the moduli spaces $\Mo$ and $\Mt$}

The bijection \eqref{eq:map} of Theorem~\ref{thm:delzant} allows us, following~\cite{PPRS14} to define a metric  on 
$\Mo$ as the pullback of the metric $\dist^V$ defined on $\Dt$  (Section~\ref{sec:metric}), thereby getting the metric space 
$(\Mo,d)$.  This metric  induces a topology $\nu$ on $\Mo$ and it makes the map of Theorem~\ref{thm:delzant}(1) a homeomorphism.

We also endow $\Mt$ with the topology $\widetilde{\nu}$ induced by~\eqref{iso moduli}, that is, with the quotient topology coming from $\Mo$, which makes Theorem~\ref{thm:delzant}(2) a homeomorphism too.

The metric $\dist^V$ is very closely related to the famous \emph{Duistermaat-Heckman measure} $m_{\rm DH}$ on $\R^n$ from \cite{DH}: 
 if $\mu \colon (M,\omega) \to \mathfrak{t}^* \simeq \R^n$ is the momentum map of a symplectic toric manifold then the 
Duistermaat\--Heckman measure on $\mathfrak{t}^* \simeq \R^n$ is given as follows:
if $U \subset \mathfrak{t}^* \simeq \R^n$ is a Borel set then $m_{\rm DH}(U)$  is (up to a constant, depending on conventions)  the  volume of $\mu^{-1}(U)$ with respect to $\omega^n$, that is, the \emph{symplectic volume} 
$\int_{\mu^{-1}(U)} \frac{\omega^n}{n!}$. Because the toric case is so especial, in terms of the Euclidean volume (Lebesgue measure on $\R^n$)
we have that $m_{\rm DH}(U)=\Vol(U \cap \mu(M))$ (in general for Hamiltonian actions of tori of lower dimension $m<n$, the expression of $m_{\rm DH}$ is
given by the integral of a function which is a polynomial of degree at most $n-m$ on each component of the regular values of $\mu$, this being a
famous result by Duistermaat\--Heckman).

\subsubsection{Main Theorem describing the geometry and topology of $\Mt$ and $\Mo$}

It follows from the work we have described in this section  that Theorems~\ref{thm:connected} and \ref{thm:quotientD} are equivalent to the following.

\begin{theorem}
\label{key}
Let $\Mo$ and $\Mt$ be the moduli spaces of toric  $2n$\--dimensional manifolds, under isomorphisms and 
equivariant isomorphisms, respectively.  Then:
\begin{itemize}
\item[(i)] $ \Mo$ and $\Mt$ are path connected.
\item[(ii)] $\Mo$ is not locally compact, but it is dense in $\CC_p(n)$ in the sense of identifying $\Mo$ with $\Dt$ via 
\eqref{eq:map}. Hence, its completion is the same as that of $\CC_p(n)$ (namely $\CC_p(n)\cup\{0\}$ for $\dist^V$ and $\CC(n)$ for $\dist^H$).
\item[(iii)] $(\widetilde{\Mo}, \widetilde{\dist^V})$ is not locally compact, but it is dense in $\widetilde{\CC_p(n)}$ in the sense of identifying $\widetilde{\Mo}$ with $\widetilde{\Dt}$ via 
\eqref{iso moduli}. Hence, its completion is the same as that of $\widetilde{\CC_p(n)}$, namely $\widetilde{\CC_p(n)}\cup\{0\}$.
\end{itemize}
\end{theorem}

Theorem~\ref{key} essentially solves a problem proposed
by the first author  and V\~u Ng\d oc~\cite[Problem 2.42]{PeVN2012} in 2012.

Similarly, Corollary~\ref{coro:simply2} implies:

\begin{proposition} \label{pol}
$\mathcal{M}(4)$ is simply connected.
\end{proposition}

\bibliographystyle{amsplain}
\bibliography{mybib}

\end{document}